\documentclass[a4paper,10pt,oneside, openany]{article}
\usepackage[a4paper, total={6in, 8in}]{geometry}
\usepackage[T1]{fontenc}
\usepackage[utf8]{inputenc}
\usepackage[italian, english]{babel}
\usepackage{microtype}
\usepackage{quoting}
\usepackage{amsthm, amsmath, amssymb, amscd, mathrsfs, mathtools} 
\usepackage{booktabs, caption, graphicx} 
\usepackage{pgfplots}
\usepackage{empheq}
\usepackage{braket}
\usepackage{emptypage}
\usepackage{amsmath}
\usepackage{amsfonts}
\usepackage{amssymb}
\usepackage{amsthm}
\usepackage{yfonts}
\usepackage{amscd}
\usepackage{extarrows}
\usepackage{multicol}
\usepackage{changepage}
\usepackage{graphicx}
\usepackage{textcomp}
\usepackage[all]{xy}
\usepackage{qtree}
\usepackage{lipsum}
\usepackage{gensymb}
\usepackage{mathtools}
\usepackage{mathrsfs}
\usepackage{etoolbox}
\usepackage{authblk}
\usepackage{eso-pic}
\usepackage{minitoc}
\usepackage{verbatim}
\usepackage{pictexwd,dcpic}
\usepackage{tikz-cd}
\usetikzlibrary{arrows}
\usepackage{adjustbox}
\usepackage{eulervm}
\usepackage{blindtext}
\usepackage[sorting=nty,backend=biber, maxnames=10]{biblatex}
\AtEveryBibitem{%
  \clearfield{issn} 
  \clearfield{doi} 
  \clearfield{isbn}
  \clearfield{note}
  \clearfield{eprint}
  \clearfield{primaryClass}

  \ifentrytype{online}{}{
    \clearfield{url}
  }
}
\quotingsetup{font=small}

\theoremstyle{definition}
\newtheorem{Def}{Definition}[section]

\theoremstyle{plain}
\newtheorem{prop}[Def]{Proposition}

\theoremstyle{plain}
\newtheorem{cor}[Def]{Corollary}

\theoremstyle{plain}
\newtheorem{thm}[Def]{Theorem}

\theoremstyle{plain}
\newtheorem{lemma}[Def]{Lemma}

\theoremstyle{remark}
\newtheorem{oss}[Def]{Remark}

\theoremstyle{remark}
\newtheorem{esempio}[Def]{Example}

\theoremstyle{remark}

\theoremstyle{definition}

\theoremstyle{plain}

\theoremstyle{plain}

\theoremstyle{plain}

\theoremstyle{plain}

\theoremstyle{remark}

\theoremstyle{remark}

\theoremstyle{plain}
\newtheorem{propintro}{Proposition}

\theoremstyle{plain}
\newtheorem{thmintro}[propintro]{Theorem}

\DeclareMathOperator{\Hom}{Hom}
\DeclareMathOperator{\Hh}{H}

\newcommand{\id}{\text{id}}

\newcommand{\ide}{\emph{id}}

\newcommand{\pico}[1]{\text{Pic}(\overline{#1})}

\newcommand{\ev}[0]{\text{ev}}

\newcommand{\Addresses}{{
  \bigskip
  \footnotesize

  \textsc{Dipartimento di Matematica, Universit\`a di Pisa, Italy}\par\nopagebreak
  \textit{E-mail address}: \texttt{mattia.pirani@phd.unipi.it}
}}

\title{Non-Surjectivity of the Universal Torsor Evaluation Map for Homogeneous Spaces}
\date{}
\author{Mattia Pirani}

\begin{document}

\maketitle

\vspace{-0.7cm}

\begin{abstract}
Let $K$ be a field of characteristic zero, let $G$ be a connected linear $K$-algebraic group, and let $H$ be a connected closed subgroup of $G$. Let $X_c$ be a smooth compactification of $X=G/H$, and let $Y\overset{}{\longrightarrow}X_c$ be the universal $S$-torsor with trivial fibre over the class of the identity of $G$. Colliot-Thélène and Kunyavski\u{\i} have shown that $S$ is a flasque torus, and that the evaluation map $X_c(K)\overset{}{\longrightarrow} \Hh^1(K,S)$, induced by the universal torsor, is surjective when the field $K$ is `good'; and the same is true when we restrict the evaluation map to the $K$-points of $X$. In this article, we establish that in cases where the field is not `good', surjectivity may fail when the domain is $X(K)$. We provide two concrete examples: one over a field of cohomological dimension $2$ and the other over an arithmetic field, such as $\mathbb Q_7((t))$.
\end{abstract}

\vspace{0.5cm}

\section*{Introduction}

Let $K$ be a field of characteristic zero, and let $X$ be a (smooth) $K$-algebraic variety, which means a geometrically integral, separated scheme of finite type over $K$. In \cite[Section 14]{MR833513}, Manin defined an equivalence relation on the $K$-points of $X$, called $R$-equivalence. Given two points $x, y \in X(K)$, we say that $x$ and $y$ are $R$-equivalent if there exists a sequence of points $x=x_1,x_2,\dots, x_{n-1}, x_n=y \in X(K) $ such that for all $i=1,\dots, n-1$, there is rational map
\begin{equation*}
        \begin{tikzpicture}[baseline= (a).base]
        \node[scale=1] (a) at (0,0){
        \begin{tikzcd}
       \mathbb P^1_k\arrow[r, dashed, "\phi_i"] & X 
       \end{tikzcd}};
        \end{tikzpicture}
\end{equation*}
such that $\text{Im}\left(\phi_i(\mathbb A^1_k(k))\right)\ni x_i, x_{i+1}$. One way to study the set $X(K)/R$ of classes of $R$-equivalence of $X$ is to find a suitable $S$-torsor $Y{\longrightarrow} X$ with $S$ a flasque torus. Indeed, the map induced by evaluation on $K$-points factors through $R$-equivalence:
\begin{equation*}
        \begin{tikzpicture}[baseline= (a).base]
        \node[scale=1] (a) at (0,0){
        \begin{tikzcd}
       X(K)/R \arrow[r, "\ev"] & \Hh^1(K,S) 
       \end{tikzcd}};
        \end{tikzpicture}
\end{equation*}
according to \cite[Corollary 2.6]{MR878473}. If we can prove that the map is injective and/or surjective, then we obtain an upper and/or lower bound on the number of $R$-equivalence classes. Furthermore, the injectivity is particularly interesting when the first Galois cohomology group $\Hh^1(K,S)$ is finite, for example when $K$ is a number field or a $l$-adic field (see \cite[Section 3.2]{MR2034644}). 

We fix a smooth compactification $X_c$ of $X$, whose existence is guaranteed by a classical theorem of Hironaka (see \cite{MR0199184}). Since $S$ is flasque, the $S$-torsor extends to $Y'{\longrightarrow} X_c$, by \cite[Theorem 2.2]{MR878473}, therefore there is the following commutative diagram:
\begin{equation*}
        \begin{tikzpicture}[baseline= (a).base]
        \node[scale=1] (a) at (0,0){
        \begin{tikzcd}
       X(K)/R \arrow[r, "\ev"]\arrow[d, "i"'] & \Hh^1(K,S). \\
       X_c(K)/R \arrow[ru, "\ev_c"']
       \end{tikzcd}};
        \end{tikzpicture}
\end{equation*}
The question about injectivity and/or surjectivity can be extended to all these maps. 

Colliot-Thélène and Sansuc studied this problem in the case where $X$ is an algebraic torus $T$. The key ingredient for their proof is the existence of flasque resolutions:
\begin{equation*}
        \begin{tikzpicture}[baseline= (a).base]
        \node[scale=1] (a) at (0,0){
        \begin{tikzcd}
       1 \arrow[r] & S\arrow[r] & E\arrow[r] & T\arrow[r] & 1,
       \end{tikzcd}};
        \end{tikzpicture}
\end{equation*}
where $S$ is a flasque torus and $E$ is a quasi-trivial one (see \cite[Lemma 1.1]{MR0364203} or \cite[Lemma 3]{requiv} for the existence). Then, the connecting morphism in Galois cohomology, induces a commutative diagram like the one above. They proved that the evaluation map $\text{ev}$ is an isomorphism in \cite[Theorem 2]{requiv}, and that the map $i$ is bijective in \cite[Proposition 13]{requiv}.

Later, these results were generalised for a connected linear $K$-algebraic group $G$. Firstly, in \cite[Proposition 1.8]{MR2067129}, Gille proved that $i$ is a bijection. Then, Colliot-Thélène extended flasque resolutions to connected linear $k$-algebraic groups, in \cite[Proposition-Definition 3.1]{MR2404747}. He defined a flasque resolution as an exact sequence 
\begin{equation*}
        \begin{tikzpicture}[baseline= (a).base]
        \node[scale=1] (a) at (0,0){
        \begin{tikzcd}
       1 \arrow[r] & S\arrow[r] & \Tilde{G}\arrow[r] & G\arrow[r] & 1
       \end{tikzcd}};
        \end{tikzpicture}
\end{equation*}
where $S$ is a flasque torus and $\Tilde{G}$ is a quasi-trivial, linear $K$-algebraic group, which means that the maximal toric quotient $\Tilde{G}^{\text{tor}}$ is quasi-trivial and the semi-simple part of $\Tilde{G}/\Tilde{G}^u$ is simply connected. He called a field $K$ `good' if it satisfies the following properties:
\begin{itemize}
    \item the cohomological dimension of $K$ is at most $2$,
    \item for all field extensions of $K$ index and period coincide for of all central simple algebras,
    \item the set $\Hh^1(K,G')$ is trivial for all quasi-trivial, linear $K$-algebraic groups $G'$.
\end{itemize}
Some examples of `good' fields are totally imaginary number fields, $l$-adic fields and $\mathbb C((x,y))$. Colliot-Thélène proved, in \cite[Theorem 8.4]{MR2404747}, that the evaluation map $\text{ev}$ is bijective when the field $K$ is `good'. 

Now consider the case of a homogeneous space $X=G/H$ under a connected linear $K$-algebraic group $G$ with connected stabilizer $H$. In \cite[Theorem 5.1]{MR2237268}, Colliot-Thélène and Kunyavski\u{\i} established that $\pico{X_c}$ is a flasque torus, extending previous results on algebraic tori, due to Voskresenski\u{\i} (see \cite[Section 4.6]{MR1634406}), and on connected linear $K$-algebraic groups, due to Borovoi and Kunyavski\u{\i} (see \cite[Theorem 3.2]{MR2054399}). Let $S$ be the flasque torus with character module $\pico{X_c}$. By \cite[Section 2.0 and Proposition 2.2.8]{sketch}, there exists a universal $S$-torsor $Y'{\longrightarrow}X_c$ with trivial fibre over the class of the identity, denoted as $[1_G]$. This torsor defines an evaluation map $\text{ev}_c$ like the one above. In \cite[Theorem 6.1]{MR2237268}, Colliot-Thélène and Kunyavski\u{\i} proved that this evaluation map $\text{ev}_c$ is surjective when the field $K$ is `good'. Following their proof, we can deduce that a `stronger' version of their theorem holds, in the sense that, under the same assumption of `good' field for $K$, even the evaluation map $\text{ev}$ is surjective.

The main results of this paper consist of constructing examples of homogeneous spaces for which the evaluation map $\text{ev}$ is non-surjective.

In Section \ref{dim2} we proceed with the construction of the first example. We do this over a field $K$ of cohomological dimension $2$. Furthermore, for this homogeneous space, the map is injective:

\begin{thmintro}
\label{thm1}
There exist a field $K$ of cohomological dimension $2$ that admits a central simple algebra of index at least $p^2$ and period $p$, and a homogeneous space $X=G/H$ defined over $K$ which satisfies the following property: if $Y{\longrightarrow} X$ is the restriction of the universal $S$-torsor over $X_c$ with trivial fibre over $[1_G]$, then the evaluation map 
$$X(K)/R\overset{\emph{ev}}{\longrightarrow} \Hh^1(k,S)$$
is injective, but not surjective. 
\end{thmintro}

The condition that over $K$ there is a central simple algebra of order $p$, for some prime $p$, and index greater than $p^2$ is crucial. For the construction of such an algebra, we use the `Index' theorem of Karpenko and Merkurjev (see \cite[Theorem 4.4]{MR3055773}). 

This result indicates that the `stronger' version of the theorem of Colliot-Thélène and Kunyavski\u{\i} cannot be extended to all fields of cohomological dimension $2$. The question about the non-surjectivity of $\text{ev}_c$ over a field of cohomological dimension $2$ it is still open, to the best of the author's knowledge. 

The problem with the previous example is that it is constructed over a very large field, indeed it is some kind of Merkurjev field (see Subsection \ref{sub3}). Therefore, in Section \ref{localt}, we construct an example defined over an arithmetic field of cohomological dimension $3$, such as $\mathbb Q_7((t))$. The main result will be the following:

\begin{thmintro}
\label{thm2}
Let $k$ be a $l$-adic field that contains a primitive $p$-th root of unity and does not contain a primitive $p^2$-th root of unity, for some prime $p\neq 2$. Suppose that $K=k((t))$. Then, there exists a homogeneous space $X=G/H$, defined over the field $K$, which satisfies the following property:
if $Y\overset{}{\longrightarrow} X$ is the restriction of the universal $S$-torsor over $X_c$ with trivial fibre over $[1_G]$, then
\begin{itemize}
    \item there is an isomorphism $\Hh^1(K,S)\simeq \mathbb Z/ p \mathbb Z$,
    \item the evaluation map
    $$X(K)/R\overset{\emph{ev}}{\longrightarrow} \Hh^1(K,S)$$
    is zero.
\end{itemize}
\end{thmintro}

First of all we note that, for all $p\neq 2$, there exists a $l$-adic field that satisfies the condition on the roots of unity. For instance, if $k=\mathbb Q_l$ for some prime $l$, then the polynomial $x^n-1$, where $l$ does not divide $n$, completely factors over $\mathbb Q_l$ if and only if completely factors over $\mathbb F_l$ by Hensel's Lemma. The condition that $\mathbb Q_l$ contains a primitive $p$-th root of unity but not a primitive $p^2$-th root of unity translates to the following one: $p$ divides $l-1$ but $p^2$ does not divide $l- 1$. There is always an infinite number of primes $l$ that satisfy both conditions. Indeed, by Strong Dirichlet's Theorem (see \cite[Chapter VI, Section 4.1]{Serre:1993}) the density of such primes inside all prime numbers is $1/p$. 

A natural question arising from Theorem \ref{thm2} is whether $X(K)$ has only one class of $R$-equivalence. It should be noted that $X$ is not a $K$-stably rational variety (otherwise $\Hh^1(K,S)$ would be trivial). 

\subsection*{Notations}

The letters $k$ and $K$ will denote two fields of characteristic zero. We always assume that $K$ is an extension of $k$. We fix an algebraic closure $\overline{k}$ of $k$ and we denote by $\Gamma_k$ the absolute Galois group of $k$. We do the same with $K$. The letter $p$ denotes a fixed prime. 

Let $M$ be a group of multiplicative type over $k$. We denote by $\widehat M$ the $\Gamma_k$-module of characters and by $\widehat M(k)$ the group of characters defined over $k$. 

Let $\Gamma$ be a finite group, and let $\widehat F$ be a $\Gamma$-lattice. We denote the dual lattice by $\widehat F^0$. 

\subsection*{Acknowledgements}

I warmly thank my advisors Philippe Gille and Tamás Szamuely for their support and for several and helpful discussions. I thank Cyril Demarche for leading me to consider abelian Galois cohomology for the construction of the second example. I also thank Mathieu Florence for useful suggestions and comments. The author is member of the Italian GNSAGA-INDAM.

\section{Preliminary results}
\label{primo}

In this section, we will investigate some criteria for the construction of the examples. It is important to highlight that when constructing the first example, there is no requirement for Subsection  \ref{catabe}.

\subsection{A first criterion for non-surjectivity}
\label{sezionefla}

In this subsection, we start by recalling some notions about flasque tori and flasque quasi-resolutions. At the end, we show that the construction of the homogeneous space can be reduced to the choice of a suitable stabilizer. 

The following two definitions are well-known, see \cite[Section 1 and Section 2]{requiv} for more details. 

\begin{Def}
A $\Gamma$-lattice $\widehat F$ is flasque (resp. coflasque) if $\Hh^1(\Lambda, \widehat F^0)$ (resp. $\Hh^1(\Lambda,\widehat F)$) is trivial for all subgroups $\Lambda<\Gamma$. A $k$-algebraic torus $T$ is flasque if $\widehat T$ is a $(\Gamma_k/\Gamma_{k'})$-flasque module, for some finite Galois extension $k\subset k'$.
\end{Def}

\begin{Def}
A $\Gamma$-lattice $\widehat F$ is a permutation module if it admits a $\Gamma$-invariant $\mathbb Z$-basis. A $k$-algebraic torus $T$ is quasi-trivial if $\widehat T$ is a $(\Gamma_k/\Gamma_{k'})$-permutation module, for some finite Galois extension $k\subset k'$.
\end{Def}

\begin{oss}
A $k$-algebraic torus $T$ is quasi-trivial if and only if it is isomorphic to a (finite) product of Weil restrictions for finite separable field extensions of $\mathbb G_m$. All permutation modules are both flasque and coflasque by Shapiro's Lemma.
\end{oss}

\begin{oss}
The base change of a flasque, coflasque or quasi-trivial torus is, respectively, flasque, coflasque or quasi-trivial. It could happen that the base change of a flasque or coflasque torus is quasi-trivial, for example when passing to the algebraic closure $\overline{k}$. 
\end{oss}

In \cite[Definition 1.1]{MR2404747}, Colliot-Thélène extended the definition of quasi-trivial tori to algebraic varieties:

\begin{Def}
\label{after}
Let $Y$ be a $k$-algebraic variety. We say that $Y$ is $\mathbb G_m$-quasi-trivial if the following conditions are satisfied:
\begin{itemize}
    \item the group $\text{Pic}(Y\times_k\overline{k})$  is trivial,
    \item the $\Gamma_k$-module $\overline{k}[Y]^*/\overline{k}^*$ is permutation.
\end{itemize}
\end{Def}

We consider a homogeneous space $X=G/H$ defined over a field extension $K$ of $k$. We assume that $H$ is a connected, reductive, linear $K$-algebraic group and that $G$ is a $\mathbb G_m$-quasi-trivial, linear $K$-algebraic group. We recall the following from \cite[Definition 4.5]{finto}:

\begin{Def}
Fix a point $x\in X(K)$. Let $Y\overset{}{\longrightarrow}X$ be a $S$-torsor, where $S$ is a flasque torus. We say that $Y\overset{}{\longrightarrow}X$ is a flasque quasi-resolution of $(X,x)$ if the following properties are satisfied:
\begin{itemize}
    \item $Y$ is a $\mathbb G_m$-quasi-trivial $K$-algebraic variety,
    \item the fibre over $x$ is trivial.
\end{itemize}
\end{Def}

Denote by $[1_G]$ the class of the identity of $G$ in $X$. We proceed with the construction of a specific flasque quasi-resolution of $(X,[1_G])$. Let $H^{\text{tor}}$ be the maximal toric quotient of $H$, and let 
 \begin{equation*}
        \begin{tikzpicture}[baseline= (a).base]
        \node[scale=1] (a) at (0,0){
        \begin{tikzcd}
       1 \arrow[r]& H^{\text{tor}}\arrow[r] & S_0\arrow[r] & E_0\arrow[r] & 1
       \end{tikzcd}};
        \end{tikzpicture}
\end{equation*}
be a flasque resolution of $H^{\text{tor}}$, which means that $S_0$ is a flasque torus and $E_0$ is a quasi-trivial one (see \cite[Lemma 0.6]{MR878473} for the existence of such an extension). We can take the pushforward of the $H$-torsor $G\overset{}{\longrightarrow} X$ by the composition morphism $H\rightarrow H^{\text{tor}}\rightarrow S_0$ to obtain a $S_0$-torsor $Y \overset{}{\longrightarrow} X$. In \cite[Proposition 7.4]{finto}, the author proved the following:

\begin{prop}
\label{esiste}
The $S_0$-torsor $Y\overset{}{\longrightarrow} X$ is a flasque quasi-resolution of $(X,[1_G])$. 
\end{prop}

Such a flasque quasi-resolution induces an evaluation map that factors through $R$-equivalence:
$$X(K)/R\overset{\text{ev}}{\longrightarrow}\Hh^1(K,S_0)$$
by considering the fibre over a $K$-point of the $S_0$-torsor $Y\overset{}{\longrightarrow} X$. The next proposition shows that the evaluation map is well-defined, in the sense that it does not depend on the chosen flasque quasi-resolution (see \cite[Proposition 6.16]{finto}).

\begin{prop}
\label{unico}
Let $Y_1\overset{}{\longrightarrow} X$ and $Y_2\overset{}{\longrightarrow} X$ be flasque quasi-resolutions of $(X,[1_G])$ under the flasque tori $S_1$ and $S_2$, respectively. Then, there exists a commutative diagram:
 \begin{equation*}
        \begin{tikzpicture}[baseline= (a).base]
        \node[scale=1] (a) at (0,0){
        \begin{tikzcd}
      & \Hh^1(K,S_1)\arrow[dd, "\sim"]\\
      X(K)/R\arrow[ru, "\emph{ev}_1"]\arrow[rd, "\emph{ev}_2"']\\
      & \Hh^1(K,S_2).
       \end{tikzcd}};
        \end{tikzpicture}
\end{equation*}
\end{prop}

We can state the first criterion for non-surjectivity:

\begin{prop}
\label{lemmafin1}
Consider the homogeneous space $X=G/H$ as described after Definition \ref{after}. Suppose that the group $H$ satisfies the following conditions:
\begin{itemize}
    \item The maximal toric quotient $H^{\emph{tor}}$ is a flasque torus.
    \item The map on cohomology $\Hh^1(K,H)\rightarrow \Hh^1(K,H^{\emph{tor}})$ is not surjective.
\end{itemize}
Then, the evaluation map $X(K)/R\overset{\emph{ev}}{\longrightarrow} \Hh^1(K,S)$ induced by the universal $S$-torsor with trivial fibre over $[1_G]$, is not surjective. 
\end{prop}

\begin{proof}
The restriction to $X$ of the universal torsor with trivial fibre over $[1_G]$ is a  flasque quasi resolution of $(X,[1_G])$ by \cite[Theorem 5.1]{MR2237268} and \cite[Proposition 5.1]{MR2404747}. Because of Proposition \ref{unico}, we can study the surjectivity of the evaluation map by choosing another flasque quasi-resolution. According to Proposition \ref{esiste}, the pushforward of the $H$-torsor $G\overset{}{\longrightarrow} X$ by $H\rightarrow H^{\text{tor}}$ is a flasque quasi-resolution of $(X,[1_G])$. Then, the evaluation map factors in the following way:
 \begin{equation*}
        \begin{tikzpicture}[baseline= (a).base]
        \node[scale=1] (a) at (0,0){
        \begin{tikzcd}
      X(K)\arrow[r, "\text{ev}"]\arrow[d] & \Hh^1(K,H^{\text{tor}}). \\
      \Hh^1(K,H)\arrow[ru]
       \end{tikzcd}};
        \end{tikzpicture}
\end{equation*}
The claim follows by diagram chasing.
\end{proof}

The previous criterion reduces the construction of the homogeneous space to the choice of a suitable $H$. Indeed, it is always possible to embed a linear algebraic group into a $\mathbb G_m$-quasi-trivial one, like $GL_n$ or $SL_n$.

\subsection{A second criterion for non-surjectivity}
\label{catabe}

In this subsection $K$ is a field extension of $k$. The main goal is to establish a refinement of Proposition \ref{lemmafin1}. This section is structured as follows: we begin by reviewing some standard facts about connected, reductive groups, then we introduce abelian cohomology and Borovoi's abelianization map, and we conclude by presenting some notions of derived categories. 

In this section $H$ is a connected, reductive linear $K$-algebraic group. We denote by $H^{\text{ss}}$ its derived subgroup (it is semi-simple) and by $H^{\text{tor}}$ the maximal toric quotient:
\begin{equation*}
        \begin{tikzpicture}[baseline= (a).base]
        \node[scale=1] (a) at (0,0){
        \begin{tikzcd}
      1 \arrow[r] & H^{\text{ss}} \arrow[r] & H \arrow[r] &  H^{\text{tor}} \arrow[r] & 1.
       \end{tikzcd}};
        \end{tikzpicture}
\end{equation*}
The group $H^{\text{sc}}$ is the universal covering of $H^{\text{ss}}$ (it is simply connected):
\begin{equation*}
        \begin{tikzpicture}[baseline= (a).base]
        \node[scale=1] (a) at (0,0){
        \begin{tikzcd}
      1 \arrow[r] & \mu \arrow[r] & H^{\text{sc}} \arrow[r] &  H^{\text{ss}} \arrow[r] & 1.
       \end{tikzcd}};
        \end{tikzpicture}
\end{equation*}
We denote by $\rho:H^{\text{sc}} \rightarrow H$ the composition morphism. We denote by $Z$ and $T$, respectively, the centre and a maximal torus of $H$. The groups $Z^{\text{(ss)}}$ and $T^{\text{(ss)}}$ are the intersection of $Z$ and $T$ with $H^{\text{ss}}$. They are, respectively, the centre and a maximal torus of $H^{\text{ss}}$. The groups $Z^{\text{(sc)}}$ and $T^{\text{(sc)}}$ are the preimages under $\rho$ of $Z$ and $T$, they are, respectively, the centre and a maximal torus of $H^{\text{sc}}$. The group $\mu$ is the kernel of $H^{\text{sc}}\rightarrow H^{\text{ss}}$, it is also the kernel of $T^{\text{(sc)}}\rightarrow T^{\text{(ss)}}$ and $Z^{\text{(sc)}}\rightarrow Z^{\text{(ss)}}$. The torus $H^{\text{tor}}$ is the cokernel of $T^{\text{(ss)}}\rightarrow T$ and of $Z^{\text{(ss)}}\rightarrow Z$, too:
\begin{equation*}
        \begin{tikzpicture}[baseline= (a).base]
        \node[scale=1] (a) at (0,0){
        \begin{tikzcd}
      1 \arrow[r] & \mu \arrow[r] & Z^{\text{(sc)}} \arrow[r] & Z \arrow[r] & H^{\text{tor}} \arrow[r] & 1.
       \end{tikzcd}};
        \end{tikzpicture}
\end{equation*}

Now, we see some fundamental concepts of abelian Galois cohomology of $H$, which are discussed in more detail in \cite{borovoi1998abelian}. In particular, we are interested to factorise $\Hh^1(K,H)\rightarrow \Hh^1(K, H^{\text{tor}})$ via the abelianization map.   

Given a group $H$ as described above, we can construct the $\Gamma_K$ complex $\left(T^{\text{(sc)}}(\overline{K})\rightarrow T(\overline{K})\right)$, where $T(\overline{K})$ is placed in degree $0$. The abelian Galois cohomology of $H$, denoted as $\Hh^i_{\text{ab}}(K, H)$, is defined as the Galois hypercohomology $ \mathbb H^i(K, T^{\text{(sc)}}(\overline{K})\rightarrow T(\overline{K}))$, for all $i\geq 1$ (see \cite[Chapter XVII]{MR0077480}).

The following result from \cite[Lemma 2.4.1]{borovoi1998abelian} demonstrates that the abelian Galois cohomology is well-defined, in the sense that it does not depend on the choice of the maximal torus $T$.

\begin{lemma}
The morphism of complexes
\begin{equation*}
        \begin{tikzpicture}[baseline= (a).base]
        \node[scale=1] (a) at (0,0){
        \begin{tikzcd}
      Z^{\emph{(sc)}}(\overline{K}) \arrow[r]\arrow[d] & Z(\overline{K}) \arrow[d]\\
      T^{\emph{(sc)}}(\overline{K}) \arrow[r] & T(\overline{K})
       \end{tikzcd}};
        \end{tikzpicture}
\end{equation*}
is a quasi-isomorphism. Consequently, it induces an isomorphism:
$$\mathbb H^i(K, Z^{\emph{(sc)}}(\overline{K})\rightarrow Z(\overline{K}))\overset{\sim}{\longrightarrow} \Hh^i_{\emph{ab}}(K,H).$$
\end{lemma}

\begin{oss}
Abelian Galois cohomology is functorial in $H$ (see \cite[Section 2.5 and Proposition 2.8]{borovoi1998abelian}). 
\end{oss}

To define the abelianization map, we need an equivalent definition for abelian Galois cohomology in lower degrees. That is achieved by utilising the concept of crossed module, which extends the notion of a $2$-term complex to the non-abelian case. We will not provide the definition of crossed modules here. Interested reader can refer to \cite[Section 3.2]{borovoi1998abelian}. Crossed modules allows us to define the pointed set $\mathbb H^1(K, H^{\text{sc}}(\overline{K})\rightarrow H(\overline{K}))$ (see \cite[Lemma 3.7.1]{borovoi1998abelian}). 

\begin{oss}
The two complexes $\left (Z^{\text{(sc)}}(\overline{K})\rightarrow Z(\overline{K})\right )$ and $\left (T^{\text{(sc)}}(\overline{K})\rightarrow T(\overline{K})\right )$, being abelian, have a trivial structure as crossed module, and these two cohomology notions coincide for them. 
\end{oss}

The following lemma demonstrates that the pointed set $\mathbb H^1(K, H^{\text{sc}}(\overline{K})\rightarrow H(\overline{K}))$ is the abelian Galois cohomology $\Hh^1_{\text{ab}}(K,H)$.

\begin{prop}
\label{quasiiso}
The canonical diagram
\begin{equation*}
        \begin{tikzpicture}[baseline= (a).base]
        \node[scale=1] (a) at (0,0){
        \begin{tikzcd}
      Z^{\emph{(sc)}}(\overline{K}) \arrow[r]\arrow[d] & Z(\overline{K}) \arrow[d]\\
      T^{\emph{(sc)}}(\overline{K}) \arrow[r]\arrow[d] & T(\overline{K})\arrow[d]\\
      H^{\emph{sc}}(\overline{K}) \arrow[r] & H(\overline{K})
       \end{tikzcd}};
        \end{tikzpicture}
\end{equation*}
is a quasi-isomorphism of crossed modules and it induces an isomorphism in cohomology (in degree $-1$, $0$ and $1$).
\end{prop}

\begin{proof}
The first part is \cite[Lemma 2.4.1]{borovoi1998abelian} and the second part is \cite[Theorem 3.3]{Borovoi1992NonabelianHO}. 
\end{proof}

The fact that we can express $\Hh^1_{\text{ab}}(K,H)$ in terms of the crossed module $\left (H^{\text{sc}}(\overline{K})\rightarrow H(\overline{K}) \right )$ leads to the definition of the abelianization map
$$\text{ab}^i:\Hh^i(K,H)\rightarrow \Hh^i_{\text{ab}}(K, H)$$
in the case of $i=0,1$. As outlined in \cite[Section 3.10]{borovoi1998abelian}, this map is induced by the following morphism of crossed modules:
\begin{equation*}
        \begin{tikzpicture}[baseline= (a).base]
        \node[scale=1] (a) at (0,0){
        \begin{tikzcd}
      1 \arrow[r]\arrow[d] & H(\overline{K}) \arrow[d]\\
      H^{\text{sc}}(\overline{K}) \arrow[r] & H(\overline{K}).
       \end{tikzcd}};
        \end{tikzpicture}
\end{equation*}

\begin{oss}
If $H$ is commutative (equivalently $H$ is a torus), then the abelianization map is the identity.
\end{oss}

\begin{oss}
The abelianization maps are functorial in $H$ (see \cite[Proposition 3.11]{borovoi1998abelian}). 
\end{oss}

The above remarks lead to the following lemma:

\begin{lemma}
\label{fattorizza}
Let $H\rightarrow H^{\emph{tor}}$ be the map from a group to its maximal toric quotient. Then, the following diagram commutes:
\begin{equation*}
        \begin{tikzpicture}[baseline= (a).base]
        \node[scale=1] (a) at (0,0){
        \begin{tikzcd}
      \Hh^1(K,H)\arrow[r]\arrow[d, "\emph{ab}^1"'] & \Hh^1(K,H^{\emph{tor}})\arrow[d, "\ide"]\\
       \Hh_{\emph{ab}}^1(K,H)\arrow[r] & \Hh^1_{\emph{ab}}(K,H^{\emph{tor}})
       \end{tikzcd}};
        \end{tikzpicture}
\end{equation*}
where $\Hh^1_{\emph{ab}}(K,H)\rightarrow \Hh^1_{\emph{ab}}(K,H^{\emph{tor}})$ is induced by the morphism of complexes:
\begin{equation*}
        \begin{tikzpicture}[baseline= (a).base]
        \node[scale=1] (a) at (0,0){
        \begin{tikzcd}
      Z^{\emph{(sc)}}(\overline{K}) \arrow[r]\arrow[d] & Z(\overline{K}) \arrow[d]\\
     1 \arrow[r] & H^{\emph{tor}}(\overline{K}).
       \end{tikzcd}};
        \end{tikzpicture}
\end{equation*}
\end{lemma}

\begin{proof}
We follow the proof of \cite[Proposition 3.11]{borovoi1998abelian}. There is the following commutative diagram of crossed modules:
\begin{equation*}
        \begin{tikzpicture}[baseline= (a).base]
        \node[scale=1] (a) at (0,0){
        \begin{tikzcd}
      \left (1 \rightarrow H(\overline{K}) \right ) \arrow[r]\arrow[d]& \left (1 \rightarrow H^{\text{tor}}(\overline{K}) \right ) \arrow[d, "\id"]\\
      \left (H^{\text{sc}}(\overline{K}) \rightarrow H(\overline{K}) \right ) \arrow[r] & \left (1 \rightarrow H^{\text{tor}}(\overline{K}) \right ).
       \end{tikzcd}};
        \end{tikzpicture}
\end{equation*}
Such diagram induces in cohomology the one of the statement. The rest of the claim follows by Proposition \ref{quasiiso}.
\end{proof}

The previous lemma provides a method for factoring the morphism $\Hh^1(K,H)\rightarrow \Hh^1(K,H^{\text{tor}})$. This motivates the study of the morphism of complexes: 
\begin{equation*}
        \begin{tikzpicture}[baseline= (a).base]
        \node[scale=1] (a) at (0,0){
        \begin{tikzcd}
      Z^{\text{(sc)}}(\overline{K}) \arrow[r]\arrow[d] & Z(\overline{K}) \arrow[d]\\
     1 \arrow[r] & H^{\text{tor}}(\overline{K}).
       \end{tikzcd}};
        \end{tikzpicture}
\end{equation*}
In the last part of this section we use some concepts related to derived categories for which we refer to \cite{MR1269324}. 

As a consequence of our initial remarks in this section, there is the following exact sequence of groups of multiplicative type:
\begin{equation*}
        \begin{tikzpicture}[baseline= (a).base]
        \node[scale=1] (a) at (0,0){
        \begin{tikzcd}
      1 \arrow[r] & \mu(\overline{K}) \arrow[r] & Z^{\text{(sc)}}(\overline{K}) \arrow[r] & Z(\overline{K}) \arrow[r] & H^{\text{tor}}(\overline{K}) \arrow[r] & 1.
       \end{tikzcd}};
        \end{tikzpicture}
\end{equation*}
We can split the sequence in two exact sequences:
\begin{equation*}
        \begin{tikzpicture}[baseline= (a).base]
        \node[scale=1] (a) at (0,0){
        \begin{tikzcd}
      1 \arrow[r] & \mu(\overline{K}) \arrow[r] & Z^{\text{(sc)}}(\overline{K}) \arrow[r] & Z ^{\text{(ss)}}(\overline{K})\arrow[r] & 1\\
       1 \arrow[r] & Z ^{\text{(ss)}}(\overline{K}) \arrow[r] & Z(\overline{K}) \arrow[r] & H^{\text{tor}}(\overline{K}) \arrow[r] & 1.
       \end{tikzcd}};
        \end{tikzpicture}
\end{equation*}
Each exact sequence induces an exact triangle in the derived category of $\Gamma_K$-modules (see \cite[Section 1.3.2]{AST_1996__239__R1_0}):
\begin{equation*}
        \begin{tikzpicture}[baseline= (a).base]
        \node[scale=1] (a) at (0,0){
        \begin{tikzcd}
        \mu(\overline{K})[1] \arrow[r] & \left ( Z^{\text{(sc)}}(\overline{K})\rightarrow Z(\overline{K})\right ) \arrow[r] & H^{\text{tor}}(\overline{K}) \arrow[r, "\delta"] & \mu(\overline{K})[2],\\
        \mu(\overline{K}) \arrow[r] & Z^{\text{(sc)}}(\overline{K}) \arrow[r] & Z^{\text{(ss)}}(\overline{K}) \arrow[r, "\delta_1"] & \mu(\overline{K})[1],\\        
      Z^{\text{(ss)}}(\overline{K}) \arrow[r] & Z(\overline{K}) \arrow[r] & H^{\text{tor}}(\overline{K}) \arrow[r, "\delta_2"] & Z^{\text{(ss)}}(\overline{K})[1].    
      \end{tikzcd}};
        \end{tikzpicture}
\end{equation*}

If we apply the Galois cohomology functor $\Hh^0(K,-)$ to the above exact triangles we get the following exact sequences:
\begin{equation*}
        \begin{tikzpicture}[baseline= (a).base]
        \node[scale=1] (a) at (0,0){
        \begin{tikzcd}
    \Hh^1_{\text{ab}}(K,H)\arrow[r] & \Hh^1(K,H^{\text{tor}}) \arrow[r, "\delta_*"] & \Hh^3(K,\mu),\\
    \Hh^2(K,Z^{\text{(sc)}}) \arrow[r] & \Hh^2(K,Z^{\text{(ss)}}) \arrow[r, "(\delta_1)_*"] & \Hh^3(K,\mu),\\
    \Hh^1(K,Z) \arrow[r] & \Hh^1(K,H^{\text{tor}}) \arrow[r, "(\delta_2)_*"] & \Hh^2(K,Z^{\text{(ss)}}).
       \end{tikzcd}};
        \end{tikzpicture}
\end{equation*}

\begin{oss}
We note that the first map of the first exact sequence is the morphism in Lemma \ref{fattorizza}.
\end{oss}

We can state the second criterion for non-surjectivity. We embed $H$ into a $\mathbb G_m$-quasi-trivial, linear $K$-algebraic group $G$ and we consider the homogeneous space $G/H$. 

\begin{prop}
\label{commuta}
Consider the homogeneous space $X=G/H$ as above. Suppose that the group $H$ satisfies the following conditions:
\begin{itemize}
    \item The maximal toric quotient $H^{\emph{tor}}$ is a flasque torus.
    \item The composite map on cohomology $\Hh^1(K,H^{\emph{tor}})\xrightarrow{(\delta_2)_*} \Hh^2(K,Z^{(\emph{ss})})\xrightarrow{(\delta_1)_*} \Hh^3(K,\mu)$ is non-trivial. 
\end{itemize}
Then, the evaluation map $X(K)/R\overset{\emph{ev}}{\longrightarrow}\Hh^1(K,S)$ induced by the universal $S$-torsor with trivial fibre over $[1_G]$, is not surjective. Furthermore, if $\Hh^1(K,H^{\emph{tor}})$ is cyclic, then the evaluation map is zero.
\end{prop}

\begin{proof}
As previously mentioned, there is the following exact sequence:
\begin{equation*}
        \begin{tikzpicture}[baseline= (a).base]
        \node[scale=1] (a) at (0,0){
        \begin{tikzcd}
    \Hh^1_{\text{ab}}(K,H)\arrow[r] & \Hh^1(K,H^{\text{tor}}) \arrow[r, "\delta_*"] & \Hh^3(K,\mu).
       \end{tikzcd}};
        \end{tikzpicture}
\end{equation*}
According to \cite[Proposition 3.2.5]{AST_1996__239__R1_0}, the map $\delta$ coincides, up to a sign, with the composition $\delta_1[1]\circ \delta_2$. Then, we can deduce that the morphism $\delta_*:\Hh^1(K,H^{\text{tor}}){\longrightarrow} \Hh^3(K,\mu)$ is non-zero and that $\Hh^1_{\text{ab}}(K,H)\rightarrow \Hh^1(K,H^{\text{tor}})$ is not surjective. The evaluation map is not surjective because of Proposition \ref{lemmafin1}. If $\Hh^1(K,H^{\text{tor}})$ is cyclic, then $\Hh^1_{\text{ab}}(K,H)\rightarrow \Hh^1(K,H^{\text{tor}})$ must be zero, so the evaluation map is zero as well. 
\end{proof}

\section{Construction of the homogeneous space when $\text{cd}(K)=2$}
\label{dim2}

In this section, we are going to conclude the construction of the homogeneous space over a field of cohomological dimension $2$. Thanks to the observation made at the end of Section \ref{sezionefla}, it is sufficient to find a suitable connected, linear $K$-algebraic group $H$ as in Proposition \ref{lemmafin1}. The letter $p$ denotes a fixed prime.

The section is organised as follows. First of all, we reduce the construction of $H$ to the existence of an exact sequence with `nice' properties. In Subsections \ref{sub1} and \ref{sub2}, we demonstrate the existence of this exact sequence, and in Subsection \ref{sub3} we prove that the same construction is valid over some field extension of cohomological dimension $2$. 

In the first part, we only assume that the field $K$ has characteristic zero. 

\begin{prop}
\label{main1}
Consider an exact sequence of groups of multiplicative type over $K$:
\begin{equation*}
        \begin{tikzpicture}[baseline= (a).base]
        \node[scale=1] (a) at (0,0){
        \begin{tikzcd}
       1 \arrow[r] & \mu_p\arrow[r, "i_1"] & R \arrow[r, "p_1"]& F\arrow[r] & 1,
       \end{tikzcd}};
        \end{tikzpicture}
\end{equation*}
that satisfies the following properties:
\begin{itemize}
    \item the group $F$ is an algebraic torus;
    \item the image of the connecting morphism $\Hh^1(K,F)\rightarrow \Hh^2(K,\mu_p)$ contains the class of a central simple algebra of index greater than $p$.
\end{itemize}
Then, there exists a connected, reductive, linear $K$-algebraic group $H$ such that:
\begin{itemize}
    \item the maximal toric quotient of $H$ is $F$;
    \item the map on cohomology $\Hh^1(K,H)\rightarrow \Hh^1(K,F)$ is not surjective.
\end{itemize}
\end{prop}

\begin{proof}
We consider the following exact sequence:
\begin{equation*}
        \begin{tikzpicture}[baseline= (a).base]
        \node[scale=1] (a) at (0,0){
        \begin{tikzcd}
       1 \arrow[r] & \mu_p\arrow[r, "i_2"] & SL_p \arrow[r, "p_2"]& PSL_p\arrow[r] & 1.
       \end{tikzcd}};
        \end{tikzpicture}
\end{equation*}
We define the linear $K$-algebraic group $H$ to fit into the following exact sequence:
\begin{equation*}
        \begin{tikzpicture}[baseline= (a).base]
        \node[scale=1] (a) at (0,0){
        \begin{tikzcd}
       1 \arrow[r] & \mu_p\arrow[r, "i_3"] & SL_p\times_K R \arrow[r, "p_3"]& H\arrow[r] & 1,
       \end{tikzcd}};
        \end{tikzpicture}
\end{equation*}
where the formula (on points) for $i_3$ is given by
\begin{equation*}
        \begin{tikzpicture}[baseline= (a).base]
        \node[scale=1] (a) at (0,0){
        \begin{tikzcd}
      a \arrow[r] & (i_2(a), i_1(a)^{-1}).
       \end{tikzcd}};
        \end{tikzpicture}
    \end{equation*}
\begin{itemize}
    \item The group $H$ is connected, reductive and its maximal toric quotient is $F$. The canonical inclusions of group schemes $j_1:R\overset{}{\longrightarrow} H$ and $j_2:SL_p \overset{}{\longrightarrow}H$ fit in the following commutative diagram, where all the sequences are exact:
    \begin{equation*}
        \begin{tikzpicture}[baseline= (a).base]
        \node[scale=1] (a) at (0,0){
        \begin{tikzcd}
       & 1 \arrow[d] & 1\arrow[d] \\
       1 \arrow[r] & \mu_p\arrow[d, "i_2"]\arrow[r, "i_1"] & R\arrow[d, "j_1"] \arrow[r, "p_1"]& F\arrow[d, "\id"]\arrow[r] & 1 \\
       1 \arrow[r]& SL_p\arrow[r, "j_2"] \arrow[d, "p_2"]& H \arrow[r, "\pi_1"]\arrow[d, "\pi_2"]& F\arrow[r] & 1 \\
       & PSL_p\arrow[d]\arrow[r, "\id"] & PSL_p\arrow[d] \\
       & 1 & 1. 
       \end{tikzcd}};
        \end{tikzpicture}
    \end{equation*}
    Since both $SL_p$ and $F$ are reductive and connected, $H$ has the same properties as well (see \cite[Proposition 5.59]{MR3729270} and \cite[Corollary 14.7]{MR3729270}). Furthermore, because $F$ is commutative and that $SL_p$ is semi-simple let us deduce that the maximal toric quotient of $H$ is $F$.
    \item The map $(\pi_1)_*:\Hh^1(K,H)\overset{}{\longrightarrow} \Hh^1(K,F)$ is not surjective. Consider the following maps:
    $$\text{H}^1(K, H)\xrightarrow{(\pi_1)_*} \text{H}^1(K,F) \overset{\delta_1}{\longrightarrow}\text{H}^2(K, \mu_p)$$
    and
    $$\text{H}^1(K, H)\xrightarrow{(\pi_2)_*} \text{H}^1(K,PSL_p) \overset{\delta_2}{\longrightarrow} \text{H}^2(K, \mu_p)$$
    First of all, we prove by a computation that they coincide, up to a sign, with the connecting map induced by the third exact sequence:  
    $$\text{H}^1(K, H) \overset{\delta_3}{\longrightarrow} \text{H}^2(K, \mu_p).$$
Given a cocycle $d$ in $Z^1(K, H(\overline{K}))$, let $(b_2, b_1):\Gamma_K \rightarrow SL_p(\overline{K})\times F(\overline{K})$ be a function that lifts $d$. If we identify $\mu_p$ with its image by $i_2$ and $i_1$, then $\delta_3 ([d])$, by construction (see \cite[Section 5.6]{MR1324577}), is the class of the following cocycle:
$$(b_2)_{\gamma_1}\cdot\  ^{\gamma_1}(b_2)_{\gamma_2} \cdot (b_2)_{\gamma_1 \gamma_2}^{-1}=(b_1)_{\gamma_1}^{-1}\cdot\  ^{\gamma_1}(b_1)_{\gamma_2}^{-1} \cdot (b_1)_{\gamma_1 \gamma_2}.$$
By diagram chasing, $\pi_2 \circ d = p_2 \circ b_2\in \text{Z}^1(K,PSL_p)$ and $\pi_1\circ d = p_1\circ b_1\in \text{Z}^1(K,F)$. Then, $b_2$ and $b_1$ lift the cocycles $(\pi_2)_* (d)$ and $(\pi_1)_* (d)$, respectively. By definition, 
$$(b_2)_{\gamma_1}\cdot\  ^{\gamma_1}(b_2)_{\gamma_2} \cdot (b_2)_{\gamma_1 \gamma_2}^{-1}$$
and
$$(b_1)_{\gamma_1}\cdot\  ^{\gamma_1}(b_1)_{\gamma_2} \cdot (b_1)_{\gamma_1 \gamma_2}^{-1}$$
correspond, respectively, to the cocycles $\delta_2 \circ (\pi_2)_* (d)$ and $\delta_1\circ (\pi_1)_* (d)$. The first one coincides with $\delta_3 ([d])$ and the second one with $-\delta_3([d])$. Suppose, by contradiction, that the map $\Hh^1(K,H)\rightarrow \Hh^1(K,F)$ is not surjective. The image of $\Hh^1(K,PSL_p)\rightarrow \Hh^2(K,\mu_p)$ consists of classes of central simple algebras with index dividing $p$ (see \cite[Theorem 2.4.3 and Proposition 2.7.9]{gille_szamuely_2006}). By hypothesis, the image of the map $\Hh^1(K,F)\rightarrow \Hh^2(K,\mu_p)$ contains the class of a central simple algebra with index greater than $p$ in its image. This, along with the following diagram:
\begin{equation*}
        \begin{tikzpicture}[baseline= (a).base]
        \node[scale=1] (a) at (0,0){
        \begin{tikzcd}
       \Hh^1(K,H) \arrow[r,"(\pi_1)_*"]\arrow[d,"(\pi_2)_*"] & \Hh^1(K,F)\arrow[d, "\delta_1"]\\
       \Hh^1(K,PSL_p) \arrow[r, "\delta_2"]& \Hh^2(K,\mu_p),
       \end{tikzcd}};
        \end{tikzpicture}
    \end{equation*}
which commutes up to a sign, leads to a contradiction. 
\end{itemize}
\end{proof}

\begin{oss}
In the previous proposition, if we assume that $F$ is a flasque torus, then we obtain the hypotheses of Proposition \ref{lemmafin1}.
\end{oss}

\subsection{The `Index' of an exact sequence}
\label{sub1}

In \cite{MR3055773}, Karpenko and Merkurjev studied the essential dimension of $p$-groups. In particular, they prove that the essential dimension of a finite $p$-group coincides with the smallest possible dimension of a faithful representation of the group. To achieve this goal, they established the Index Theorem, which can be found in \cite[Theorem 4.4]{MR3055773} or \cite[Section 6]{MR2393078}. Our interest in this theorem lies in its potential to enable the construction of some specific central simple algebra with index greater than $p$. In the first part of this section, we will recall some key definitions and state the Index Theorem. In the second part, we will study the `index' of an exact sequence and apply the results to the construction of an exact sequence as in Proposition \ref{main1}.

Consider an exact sequence of groups of multiplicative type over the field $k$
\begin{equation}
\tag{e}
        \begin{tikzpicture}[baseline= (a).base]
        \node[scale=1] (a) at (0,0){
        \begin{tikzcd}
       1 \arrow[r] & M_1\arrow[r, "\alpha"] & M \arrow[r]& M_2\arrow[r] & 1. 
       \end{tikzcd}};
        \end{tikzpicture}
\end{equation}

Let $\mathcal E \in \Hh^1(\Tilde{k},M_2)$ for some field extension $k\subset \Tilde{k}$. We define a group morphism $\beta^{\mathcal E} : \widehat M_1 (k) \rightarrow \text{Br}(\Tilde{k})$ as follows: for all $\chi \in \widehat M_1 (k)$, the class $\beta^{\mathcal E}(\chi)$ is the image of $\mathcal E$ by the composition morphism:
\begin{equation*}
        \begin{tikzpicture}[baseline= (a).base]
        \node[scale=1] (a) at (0,0){
        \begin{tikzcd}
      \text{H}^1(\Tilde{k},M_2) \arrow[r, "\delta"]&  \text{H}^2(\Tilde{k}, M_1) \arrow[r, "\chi_*"]& \text{H}^2(\Tilde{k},\mathbb G_m)=\text{Br}(\Tilde{k}).
       \end{tikzcd}};
        \end{tikzpicture}
\end{equation*}
We also define the morphism
\begin{equation*}
        \begin{tikzpicture}[baseline= (a).base]
        \node[scale=1] (a) at (0,0){
        \begin{tikzcd}
       \Hh^1(\Tilde{k},M_2) \arrow[r, "\beta"]& \text{Hom}_{\text{Ab}}(\widehat M_1 (k), \text{Br}(\Tilde{k})) 
       \end{tikzcd}};
        \end{tikzpicture}
\end{equation*}
that maps $\mathcal E$ to $\beta^{\mathcal E}$. 

\begin{oss}
This morphism is functorial in $\Tilde{k}$.
\end{oss}

We fix a character $\chi\in \widehat M_1 (k)$. Let $\text{Rep}^{(\chi)}(e)$ be the set of representations $\rho: M\rightarrow GL_n$ that fit into the following commutative diagram:
\begin{equation*}
        \begin{tikzpicture}[baseline= (a).base]
        \node[scale=1] (a) at (0,0){
        \begin{tikzcd}
      M_1 \arrow[r, "\alpha"]\arrow[d, "\chi"]&  M \arrow[d, "\rho"]\\
      \mathbb G_m\arrow[r, "\Delta"] & GL_n,
       \end{tikzcd}};
        \end{tikzpicture}
\end{equation*}
where $\Delta:\mathbb G_m\rightarrow GL_n$ is the canonical embedding of diagonal matrices. The index (relative to $\chi$) of the exact sequence $e$ is defined as $$I(\chi,e )\coloneqq \text{gcd}\Bigl\{n\ |\ \rho:M\rightarrow GL_n \in \text{Rep}^{(\chi)}(e)\Bigr\}.$$

To prove that $I(\chi,e)$ is well-defined, we need to show that the set $\text{Rep}^{(\chi)}(e)$ is not-empty. Before proceeding, we require a well-known lemma:

\begin{lemma}
\label{lemmadacitare}
Let $A=k_1\times \cdots \times k_r$ be a $n$-dimensional étale algebra over $k$. Then:
\begin{itemize}
    \item The group morphism
    \begin{equation*}
        \begin{tikzpicture}[baseline= (a).base]
        \node[scale=1] (a) at (0,0){
        \begin{tikzcd}
      R_{k_i/k}(\mathbb G_m)\arrow[r] &GL_{k}(k_i)
       \end{tikzcd}};
        \end{tikzpicture}
    \end{equation*}
given by the formula (on points)
    \begin{equation*}
        \begin{tikzpicture}[baseline= (a).base]
        \node[scale=1] (a) at (0,0){
        \begin{tikzcd}x \arrow[r] & (m_x:v\longmapsto v\cdot x)
       \end{tikzcd}};
        \end{tikzpicture}
    \end{equation*}
    is an embedding. Precomposition with the diagonal embedding $\mathbb G_m \rightarrow R_{k_i/k}(\mathbb G_m)$ gives the canonical embedding into the diagonal matrices. Conversely, the embedding of diagonal matrices factors through the above morphisms. 
    \item The composition
    $$R_{A/k}(\mathbb G_m)=R_{k_1/k}(\mathbb G_m)\times \cdots \times R_{k_r/k}(\mathbb G_m)\rightarrow GL_k(k_1)\times \cdots \times GL_k(k_r)\rightarrow GL_k(A)\simeq GL_n$$
    defines a maximal torus of $GL_n$. Up to conjugation, all maximal tori of $GL_n$ are of this form.
\end{itemize}
\end{lemma}

\begin{proof}
See \cite[Section 2]{MR3039772} and \cite[Section 6.1]{MR1634406}.
\end{proof}

The following proposition has two implications. First of all, it will serve to establish that $\text{Rep}^{(\chi)}(e)$ is not-empty, and then, it will help to study $I(\chi,e)$.

\begin{prop}
\label{repsplit}
Consider the following properties:
\begin{enumerate}
    \item There exists a degree $n$ extension of $k$ that splits $\chi_*(e)$.
    \item There exists a representation $M\rightarrow GL_n \in \emph{Rep}^{(\chi)}(e)$.
\end{enumerate}
The implication $1 \implies 2$ always holds. The implication $2\implies 1$ holds if we assume the existence of an irreducible representation.
\end{prop}

\begin{proof}
Extension of scalars and Weil restriction are adjoint functors. Therefore, for a field extension $k\subset k'$, there is the following commutative diagram 
\begin{equation*}
        \begin{tikzpicture}[baseline= (a).base]
        \node[scale=1] (a) at (0,0){
        \begin{tikzcd}
       M_1 \arrow[r, "\alpha"]\arrow[d, "\chi"]& M\arrow[d, dashed]\\
      \mathbb G_{m} \arrow[r, "\Delta"]& R_{k'/k}(\mathbb G_{m,k'}),
       \end{tikzcd}};
        \end{tikzpicture}
    \end{equation*}
if and only if there is the following one
\begin{equation*}
        \begin{tikzpicture}[baseline= (a).base]
        \node[scale=1] (a) at (0,0){
        \begin{tikzcd}
       M_{1,k'} \arrow[r, "\alpha_{k'}"]\arrow[d, "\chi_{k'}"]& M_{k'}\arrow[d, dashed]\\
      \mathbb G_{m,k'} \arrow[r, "\id"]& \mathbb G_{m,k'}
       \end{tikzcd}};
        \end{tikzpicture}
    \end{equation*}
if and only if $\chi_*(e)$ is split over $k'$. 

$1\implies 2$. Let $k'$ be an extension of $k$ of degree $n$ that splits $\chi_*(e)$. There is the following commutative diagram:
    \begin{equation*}
        \begin{tikzpicture}[baseline= (a).base]
        \node[scale=1] (a) at (0,0){
        \begin{tikzcd}
       M_1 \arrow[r, "\alpha"]\arrow[d, "\chi"]& M\arrow[d, "\rho"]\\
      \mathbb G_{m} \arrow[r, "\Delta"]& R_{k'/k}(\mathbb G_{m,k'}).
       \end{tikzcd}};
        \end{tikzpicture}
    \end{equation*}
    By Lemma \ref{lemmadacitare}, there exists an embedding $i:R_{k'/k}(\mathbb G_{m,k'})\rightarrow GL_n$ and the composition $i \circ \Delta $ is the canonical embedding of the diagonal matrices. By construction, the morphism $i \circ \rho$ is a representation of dimension $n$ in $\text{Rep}^{(\chi)}(e)$.

$2\implies 1$. By hypothesis, there is the following commutative diagram:
    \begin{equation*}
        \begin{tikzpicture}[baseline= (a).base]
        \node[scale=1] (a) at (0,0){
        \begin{tikzcd}
       M_1 \arrow[r, "\alpha"]\arrow[d, "\chi"]& M\arrow[d, "\rho"]\\
      \mathbb G_{m} \arrow[r, "\Delta"]& GL_n,
       \end{tikzcd}};
        \end{tikzpicture}
    \end{equation*}
    where $\rho$ is irreducible.
    Since $M$ is a group of multiplicative type, by \cite[Section 2]{MR3039772}, the representation $\rho$ factors trough a maximal torus $i:R_{A/k}(\mathbb G_{m,A})\rightarrow GL_n$, where $A$ is a $n$-dimensional étale algebra. Since $\rho$ is irreducible, the algebra $A$ is a finite extension of $k$ of degree $n$. The morphism $\mathbb G_m\rightarrow R_{A/k}(\mathbb G_{m,A})$ is the canonical embedding, by Lemma \ref{lemmadacitare}, then $\chi_*(e)$ is split over $A$.
\end{proof}

\begin{oss}
Proposition \ref{repsplit}, together with the fact that $\chi_*(e)$ splits over a finite extension of $k$, establishes that the index $I(\chi,e)$ is well-defined.
\end{oss}

\begin{cor}
The following identity holds:
$$I(e, \chi)=\emph{gcd}\Bigl\{[k':k]\ |\ k' \emph{ splits } \chi_*(e)\Bigr\}.$$
\end{cor}

\begin{proof}
First of all, we observe that there is the following identity:
$$\text{gcd}\Bigl\{n\ |\ \rho:M\rightarrow GL_n \in \text{Rep}^{(\chi)}(e)\Bigr\}=\text{gcd}\Bigl\{n\ |\ \rho:M\rightarrow GL_n \in \text{Rep}^{(\chi)}(e) \text{ irreducible}\Bigr\}.$$
Indeed, the set $\text{Rep}^{(\chi)}(e)$ is closed under sub-representations and quotients. Therefore, when we have $\rho:M\rightarrow GL_n \in \text{Rep}^{(\chi)}(e)$, we can consider an irreducible representation of $\rho$ and take the quotient. This process can be iterated using the quotient, ultimately allowing us to express $n$ as the sum of the dimensions of irreducible representations in $\text{Rep}^{(\chi)}(e)$. This implies that the greatest common divisor of irreducible representations divides the greatest common divisor of all representations in $\text{Rep}^{(\chi)}(e)$. The reverse divisibility is trivial. The statement now follows by Proposition \ref{repsplit}.
\end{proof}

Given $\rho: M\rightarrow GL_n \in \text{Rep}^{(\chi)}(e)$, there is the following commutative diagram with exact rows:
\begin{equation*}
        \begin{tikzpicture}[baseline= (a).base]
        \node[scale=1] (a) at (0,0){
        \begin{tikzcd}
      1 \arrow[r] & M_1 \arrow[r, "\alpha"] \arrow[d, "\chi"] & M \arrow[r]\arrow[d, "\rho"] & M_2 \arrow[r]\arrow[d] & 1 \\
      1 \arrow[r] & \mathbb G_m \arrow[r, "\Delta"] & GL_n \arrow[r] & PGL_n \arrow[r] & 1.
       \end{tikzcd}};
        \end{tikzpicture}
\end{equation*}
Such a diagram induces in cohomology the following commutative square, for all field extensions $k\subset \Tilde{k}$:
\begin{equation*}
        \begin{tikzpicture}[baseline= (a).base]
        \node[scale=1] (a) at (0,0){
        \begin{tikzcd}
      \Hh^1(\Tilde{k},M_2) \arrow[r]\arrow[d]&   \Hh^2(\Tilde{k},M_1) \arrow[d, "\chi_*"]\\
       \Hh^1(\Tilde{k}, PGL_n)\arrow[r] &  \Hh^2(\Tilde{k},\mathbb G_m).
       \end{tikzcd}};
        \end{tikzpicture}
\end{equation*}
This commutative square implies that $\text{ind}(\beta^{\mathcal E}(\chi))$ divides $I(\chi,e)$. 

\begin{oss}
If the extension $e$ is split or $\chi$ is trivial, then both numbers are one.
\end{oss}

We can embed the group $M_2$ into a quasi-trivial torus $P$ (see \cite[Lemma 0.6]{MR878473}):
\begin{equation*}
        \begin{tikzpicture}[baseline= (a).base]
        \node[scale=1] (a) at (0,0){
        \begin{tikzcd}
      1 \arrow[r] & M_2 \arrow[r]  & P \arrow[r] & P/M_2 \arrow[r] & 1. 
       \end{tikzcd}};
        \end{tikzpicture}        
\end{equation*}
The generic fibre of the $M_2$-torsor $P\rightarrow P/M_2$ is a generic $M_2$-torsor, in the sense of \cite[p. 425]{MR3055773}. 

\begin{oss}
\label{riferimento}
Such a generic torsor is defined over a finitely generated field extension of $k$.
\end{oss}

We can state the following theorem of Karpenko and Merkurjev (see \cite[Theorem 6.1]{MR3055773}):

\begin{thm}
\label{teoremakm}
Assume that $M_1$ is diagonalizable and suppose that ${\mathcal E}$ is the class of a generic $M_2$-torsor for a field extension $K$ of $k$. Then, $\emph{ind}(\beta^{\mathcal E}(\chi))= I(\chi,e)$. 
\end{thm}

We have just seen a significant connection between the index of a specific class of central simple algebras and $I(\chi,e)$. This, in combination with the previous study of $I(\chi,e)$, reduces the process of constructing an exact sequence as in the statement of Proposition \ref{main1} to a group cohomology argument:

\begin{prop}
\label{propkm}
Let $k$ be a $p$-special field (i.e. $\Gamma_k$ is a pro-$p$-group). We suppose that there exist a $k$-algebraic torus $F$ and an element $z\in \Hh^1(k,\widehat F)$ of order $p$, such that for all degree $p$ extensions $k\subset k'$ the restriction $z_{k'}$ is non-trivial. Then, there exists an exact sequence of groups of multiplicative type over $k$
\begin{equation*}
        \begin{tikzpicture}[baseline= (a).base]
        \node[scale=1] (a) at (0,0){
        \begin{tikzcd}
       1 \arrow[r] & \mu_p\arrow[r] & R \arrow[r]& F\arrow[r] & 1
       \end{tikzcd}};
        \end{tikzpicture}
\end{equation*}
and a finitely generated field extension $k\subset K$, such that the image of the connecting morphism $\Hh^1(K,F)\rightarrow \Hh^2(K,\mu_p)$ contains the class of a central simple algebra of index greater than $p$.
\end{prop}

\begin{proof}
There is the following chain of (functorial) isomorphisms: 
$$\Hh^1(k,\widehat F)\simeq \text{Pic}(F)\simeq \text{Ext}_k(F,\mathbb G_m),$$
where the first isomorphism is from \cite[Section 4.3]{MR0262251} and the second one is from \cite[Corollary 5.7]{MR2404747}. The cohomology class $z$ corresponds to an extension of $F$ by $\mathbb G_m$ of order $p$ that does not split over any degree $p$ extension of $k$. The exact sequence
\begin{equation*}
        \begin{tikzpicture}[baseline= (a).base]
        \node[scale=1] (a) at (0,0){
        \begin{tikzcd}
       1 \arrow[r] & \mu_p \arrow[r, "\chi"] & \mathbb G_m \arrow[r, "(-)^p"] & \mathbb G_m \arrow[r] & 1 
       \end{tikzcd}};
        \end{tikzpicture}
\end{equation*}
induces the following exact sequence:
\begin{equation*}
        \begin{tikzpicture}[baseline= (a).base]
        \node[scale=1] (a) at (0,0){
        \begin{tikzcd}
       \text{Ext}_k(F,\mu_p) \arrow[r, "{\chi}_*"] & \text{Ext}_k(F,\mathbb G_m)\arrow[r, "(-)^p"] & \text{Ext}_k(F,\mathbb G_m). 
       \end{tikzcd}};
        \end{tikzpicture}
\end{equation*}
From this, we can deduce the existence of an extension
\begin{equation}
\tag{e}
        \begin{tikzpicture}[baseline= (a).base]
        \node[scale=1] (a) at (0,0){
        \begin{tikzcd}
       1 \arrow[r] & \mu_p\arrow[r] & R \arrow[r]& F\arrow[r] & 1
       \end{tikzcd}};
        \end{tikzpicture}
\end{equation}
such that the pushforward by the canonical embedding $\chi:\mu_p\rightarrow \mathbb G_m$ does not split over any degree $p$ extension of $k$. The linear $k$-algebraic group $R$ is of multiplicative type, as it is an extension of an algebraic torus by a group of multiplicative type (see \cite[Proposition 4.1.4.5]{MR0302656}). Since $k$ is $p$-special, the greatest common divisor
$$\text{gcd}\Bigl\{[k':k]\ |\ k' \text{ splits } {\chi}_*(e)\Bigr\}.$$
is a minimum. This implies that $I(e,\chi)$ is greater than $p$ because there is no degree $p$ extension that splits ${\chi}_*(e)$. Let $K$ be a finitely generated extension of $k$ such that there exists a generic torsor $\mathcal E$ under $F$ defined over $K$ (see the construction after Remark \ref{riferimento}). By Theorem \ref{teoremakm}, the image of the class of $\mathcal E$ under the connecting morphism $\Hh^1(K,F)\rightarrow \Hh^2(K,\mu_p)$ has index $I(\chi,e)$. In particular, it is greater than $p$. 
\end{proof}

\subsection{Construction of the flasque torus}
\label{sub2}

The main goal of this subsection is to construct a flasque $\Gamma_k$-module $\widehat {F}$ and an element $z\in \Hh^1(k,\widehat {F})$ as in the statement of Proposition \ref{propkm}. 

First, we present two preliminary results. Consider a $\Gamma_k$-lattice $\widehat {F}$. We denote by $k_F$ the splitting field of $\widehat {F}$ (over $k$), i.e. the field that makes the following sequence exact:
$$1 \rightarrow \Gamma_{k_F} \rightarrow \Gamma_k \rightarrow \text{Aut}(\widehat {F}).$$

\begin{oss}
The field $k_F$ is a finite Galois extension of $k$. 
\end{oss}

We prove the following lemma:

\begin{lemma}
\label{tdg}
Let $z \in \emph{H}^1(k,\widehat {F})$ and let $k\subset l\subset \overline{k}$ be a finite Galois extension that splits $z$. Then $z$ already splits over $k_F\cap l$. 
\end{lemma}

\begin{proof}
We proceed in two steps:
\begin{itemize}
    \item Let $k\subset k'\subset \overline{k}$ be a finite field extension. The composite field of $k'$ and $k_F$ is the splitting field of $\widehat {F}$ over $k'$, that we denote by $k'_F$. Indeed, the field $k'_F$ contains $k'$ by definition. Moreover, since $\Gamma_{k'_F}$ is in the kernel of $\Gamma_k \rightarrow \text{Aut}(\widehat {F})$, the field $k'_F$ contains $k_F$. This proves that $k'_F$ contains $k'k_F$. For the opposite inclusion, we note that $\Gamma_{k'k_F}$ acts trivially on $\widehat {F}$, thus it is in the kernel of $\Gamma_{k'} \rightarrow \text{Aut}(\widehat {F})$. 
    \item By the first item, the splitting field of $\widehat {F}$ over $k_F\cap l$ is $(k_F\cap l)k_F=k_F$. Therefore, we can assume, without loss of generality, that $k_F\cap l =k$. Both extensions $k\subset l$ and $k\subset k_F$ are Galois and their intersection is $k$, so the canonical restriction map
    $$\text{Gal}(lk_F/l)\rightarrow \text{Gal}(k_F/k)$$
    is an isomorphism. Moreover, again by the first item, there is an identity $lk_F = l_F$. We have the following commutative diagram
    \begin{equation*}
        \begin{tikzpicture}[baseline= (a).base]
        \node[scale=1] (a) at (0,0){
        \begin{tikzcd}
      \text{H}^1(k,\widehat {F})\arrow[r] &  \text{H}^1(l,\widehat {F})\\
        \text{H}^1(\text{Gal}(k_F/k),\widehat {F})\arrow[r, "\sim"]\arrow[u, "\sim"] & \text{H}^1(\text{Gal}(lk_F/l),\widehat {F}).\arrow[u, "\sim"]
       \end{tikzcd}};
        \end{tikzpicture}
\end{equation*}
    Thus, the restriction map $\text{H}^1(k,\widehat {F})\rightarrow \text{H}^1(l,\widehat {F})$ is an isomorphism, which means $z_l=0$ if and only if $z=0$. 
\end{itemize}
\end{proof}

We proceed with the construction of the flasque module. Consider a finite non-trivial $p$-group $\Gamma$, and denote by $(\Lambda_i)_{i\in I}$ the family of its subgroups of index $p$. We assume that $\Lambda_i$ is not cyclic for all $i\in I$, for example $\Gamma=(\mathbb Z/ p \mathbb Z)^3$.

\begin{prop}
\label{esistenzaflasque}
There exists a flasque $\Gamma$-lattice $\widehat {F}$ that verifies the following property: there exists a $p$-torsion element $z \in \Hh^1(\Gamma,\widehat {F})$ such that, for all $i\in I$, the restriction of $z$ to $\Hh^1(\Lambda_i, \widehat F)$ is non-trivial.
\end{prop}

\begin{proof}
By \cite[Corollary 2]{requiv} we can find a flasque $\Lambda_i$-module $\widehat{F}'_i$ for which $\Hh^1(\Lambda_i', \widehat{F}'_i)\neq 0$ holds for some subgroup $\Lambda_i'<\Lambda_i$. We denote the coinduced modules $\text{CoInd}_{\Lambda_i'}^{\Gamma}(\widehat{F}'_i)$ by $\widehat F_i$. The $\Gamma$-module $\widehat F_i$ has the following properties:
    \begin{itemize}
        \item It is flasque (see \cite[Remark 2 and Remark 3]{requiv}).
        \item The group $\Hh^1(\Gamma,\widehat F_i)\simeq \Hh^1(\Lambda_i', \widehat{F}'_i)$ is non-zero.
        \item The restriction map 
        $$\Hh^1(\Gamma,\widehat F_i)\simeq \Hh^1\Bigl(\Lambda_i, \text{CoInd}_{\Lambda_i'}^{\Lambda_i}(\widehat{F}'_i)\Bigr)\overset{\Delta}{\longrightarrow} \Hh^1\Bigl(\Lambda_i, \text{CoInd}_{\Lambda_i'}^{\Lambda_i}(\widehat{F}'_i)\Bigr)^p\simeq \Hh^1(\Lambda_i,\widehat F_i)$$
        is the diagonal map. In particular, it is injective.
    \end{itemize}
Now, let $z_i$ be an element of order $p$ in $\Hh^1(\Gamma,\widehat F_i)$. We denote by $\widehat F$ the direct sum of all $\widehat F_i$, and let $z=(z_i)_{i\in I} \in \bigoplus_{i\in I} \Hh^1(\Gamma,\widehat F_i)=\Hh^1(\Gamma,\widehat F)$. By construction, the order of $z$ is $p$ and $z$ does not belong to the union:
$$\bigcup_{i\in I} \text{ker}\Bigl(\Hh^1(\Gamma,\widehat F)\rightarrow \Hh^1(\Lambda_i,\widehat F)\Bigr).$$
\end{proof}

\begin{esempio}
We provide a concrete example using the smallest possible group $\Gamma=(\mathbb Z/ p\mathbb Z)^3$, indeed all the group of cardinality dividing $p^2$ have index $p$ subgroups that are cyclic. All of its index $p$ subgroups are isomorphic to $\Lambda=(\mathbb Z/p \mathbb Z)^2$. Consider the following exact sequence:
\begin{equation*}
        \begin{tikzpicture}[baseline= (a).base]
        \node[scale=1] (a) at (0,0){
        \begin{tikzcd}
       0 \arrow[r]& \Tilde{F}^0\arrow[r] & \mathbb Z[\Lambda]^2\arrow[r, "\phi"] & \mathbb Z[\Lambda]\arrow[r, "\epsilon"] & \mathbb Z\arrow[r] & 0, 
       \end{tikzcd}};
        \end{tikzpicture}
\end{equation*}
where $\epsilon (\sum_{\lambda\in \Lambda} n_{\lambda}\lambda)=\sum_{\lambda\in \Lambda} n_{\lambda}$ and $\phi(\sum_{\lambda\in \Lambda}n_{\lambda}^1\lambda,\sum_{\lambda\in \Lambda}n_{\lambda}^2\lambda )=\sum_{\lambda\in \Lambda}n_{\lambda}^1\lambda(\lambda_1-1)+n_{\lambda}^2\lambda(\lambda_2-1)$ for two fixed generators $\lambda_1$ and $\lambda_2$ of $\Lambda$. According to \cite[Proposition 1]{requiv}, the $\Lambda$-module $\Tilde{F}$ is flasque. Additionally, there exists an isomorphism $\Hh^1(\Lambda, \Tilde{F})\simeq \Hh^3(\Lambda,\mathbb Z)\simeq \mathbb Z/p \mathbb Z$. This implies that for all $i\in I$ we can take $\Tilde{F}_i\coloneqq \Tilde{F}$. 
\end{esempio}

Fix a Galois field extension $k_0\subset l_0$ with group $\Gamma$ as above. Let $k$ be the fixed field of a pro-$p$-Sylow subgroup of $\Gamma_{k_0}$. Let $l$ be the composite field of $k$ and $l_0$. Let $\widehat F$ be a flasque $\Gamma$-module as in the statement of Proposition \ref{esistenzaflasque}. 

\begin{oss}
The extension $k\subset l$ is Galois with group $\Gamma$, $k$ is $p$-special and there is an isomorphism $\Hh^1(k,\widehat F)\simeq \Hh^1(\Gamma,\widehat F)$. Furthermore, we note that $k_F$ contains $l$. 
\end{oss}

We can state the main result of this section:

\begin{prop}
\label{finnale}
There exist a $p$-special field $k$, a flasque $\Gamma_k$-module $\widehat F$ and an element $z\in \Hh^1(k,\widehat F)$ of order $p$, such that for all degree $p$ extensions $k\subset k'\subset \overline{k}$ the restriction $z_{k'}$ is non-trivial.
\end{prop}

\begin{proof}
By the previous construction, there exist a flasque module over a $p$-special field $k$ and $z\in \Hh^1(k,\widehat F)$ of order $p$ such that $z_{k'}$ is non-trivial for all degree $p$ extensions $k\subset k'$ that are subextensions of $k\subset l$. By Lemma \ref{tdg}, this is sufficient to deduce that $z$ is non-trivial over all degree $p$ extensions of $k$.
\end{proof}

The previous proposition ends the construction of the homogeneous space over a finitely generated field extension of $k$. In the next section, we prove that this example can be realised over a field of cohomological dimension $2$.

\subsection{Reduction to a field of cohomological dimension $2$}
\label{sub3}

In the previous subsections, we have constructed an example of homogeneous space, defined over a field $K$, for which the evaluation map induced by a flasque quasi-resolution is not surjective. Now, we want to find a field extension $K^{(\infty)}$ of $K$ that satisfies the following properties:
\begin{itemize}
    \item the field $K^{(\infty)}$ has cohomological dimension at most $2$,
    \item for all central simple algebras $A$ over $K$, there is the equality $\text{ind}(A)=\text{ind}(A\otimes_K K^{(\infty)})$.
\end{itemize}
This is enough to enough to find the searched homogeneous space over a field of cohomological dimension $2$. Indeed, we can apply the field extension $K\subset K^{(\infty)}$ to the field  $K$ that appears in Proposition \ref{propkm}, that is the field of definition of a generic torsor. Indeed, by Proposition \ref{finnale} and Proposition \ref{propkm}, we have an exact sequence of groups of multiplicative type over $k$
    \begin{equation*}
        \begin{tikzpicture}[baseline= (a).base]
        \node[scale=1] (a) at (0,0){
        \begin{tikzcd}
       1 \arrow[r] & \mu_p\arrow[r] & R \arrow[r]& F\arrow[r] & 1,
       \end{tikzcd}};
        \end{tikzpicture}
    \end{equation*}
    where $F$ is a flasque torus, and we also have a finitely generated field extension $k\subset K$, such that the image of the connecting morphism $\Hh^1(K,F)\rightarrow \Hh^2(K,\mu_p)$ contains the class of a central simple algebra of index greater than $p^2$. By replacing $K$ with $K^{(\infty)}$, we obtain a field of cohomological dimension $2$ and, because of the second property mentioned earlier, the image of the connecting morphism $\Hh^1(K^{(\infty)},F)\rightarrow \Hh^2(K^{(\infty)},\mu_p)$ still contains the class of a central simple algebra of index greater than $p^2$. Now, we can use Proposition \ref{main1} and Proposition \ref{lemmafin1} to conclude the construction and to prove Theorem \ref{thm1}.

\begin{oss}
Following this construction we obtain a homogeneous space $X=G/H$ for which the evaluation map is not surjective. Furthermore, the group $H$ is the extension of a flasque torus by the group $SL_p$. This implies that if we choose $G$ in a way that $G(K)/R$ is trivial, for example, by choosing $G$ as $GL_n$, then the evaluation map is injective (see \cite[Proposition 7.11]{finto}).
\end{oss}

The idea is to use a construction of Ducros (see \cite[pp. 351-352]{MR1622783}) and to apply a theorem of Panin (see \cite[Corollary 5.2]{MR1618404}). In the rest of the subsection we only assume that $K$ has characteristic zero.

Let $\mathcal F$ be a collection of representatives of birational classes of non-trivial $R_{L/K}(SL_1(A))$-torsors over $K$, where $L$ ranges over finite extensions of $K$ and $A$ varies over central simple algebras over $L$. 

For each finite non-empty subset $\mathcal U=\{X_1, \cdots, X_r\}$ of $\mathcal F$ we define $K_{\mathcal U}$ to be the fraction field of the product $X_1\times_K \cdots \times_K X_r$. If $\mathcal U$ is the empty set, then $K_{\mathcal U}\coloneqq K$.

\begin{oss}
The field $K_{\mathcal U}$ is well defined, indeed $X_1\times_K \cdots \times_K X_r$ is a (geometrically) integral $K$-algebraic variety. 
\end{oss}

For each inclusion $\mathcal U\subset \mathcal V$ of subsets of $\mathcal F$, there is an inclusion of fields $K_{\mathcal U}\subset  K_{\mathcal V}$. Then we can take the following colimit:
$$K^{(1)}\coloneqq\lim_{\longrightarrow } K_{\mathcal U}.$$
The construction can be iterated, defining
$$K^{(n+1)}\coloneqq (K^{(n)})^{(1)}.$$
We denote by $K^{(\infty)}$ the union of all such fields:
$$K\subset K^{(1)}\subset K^{(2)}\subset K^{(3)}\subset ...$$
In the following part of this section, we will prove that $K^{(\infty)}$ satisfies the aforementioned properties.

\begin{lemma}
\label{lemmacohomo2}
Let $D$ be a division algebra over $K$. The central simple algebra $D\otimes_K K^{(\infty)}$ is also a division algebra. Consequently, if $A$ is a central simple algebra over $K$, then $\emph{ind} (A)=\emph{ind} (A\otimes_K K^{(\infty)})$.
\end{lemma}

\begin{proof}
If $D\otimes_K K^{(\infty)}$ is not a division algebra then, for sufficiently large $n$, $D\otimes_K K^{(n)}$ is not a division algebra. Therefore, it suffices to show that $D\otimes_K K^{(1)}$ is a division algebra. Let us suppose, by contraction, that $D\otimes_K K^{(1)}$ is not a division algebra. This implies the existence of a finite subset $\mathcal U \subset \mathcal F$ such that $D\otimes_K K_{\mathcal U}$ is not a division algebra. We can further assume that $\mathcal U$ is not empty; otherwise, the contradiction is immediate. The field $K_{\mathcal U}$ is the quotient field of a principal homogeneous space under $R_{K_1/K}(SL_1(A_1))\times_K \cdots \times_K R_{K_r/K}(SL_1(A_r))$ for some finite extensions $K\subset K_i$ and some central simple algebras $A_i$ over $K_i$. Over the algebraic closure $\overline{K}$, this group is isomorphic to a product of $SL_{n_j}$, thus by Galois descent, it is semi-simple and simply connected. According to  Panin's theorem (see \cite[Corollary 5.2]{MR1618404}), the central simple algebra $D\otimes_K K_{\mathcal U}$ must be a division algebra. This contradiction shows that $D\otimes_K K^{(1)}$ is indeed a division algebra. Consequently, $\text{ind}(A)=\text{ind}(A\otimes_K K^{(\infty)})$ as required.
\end{proof}

The following proposition was already stated in \cite{MR1622783}.

\begin{prop}
\label{propcohomo2}
The field $K^{(\infty)}$ has cohomological dimension at most $2$. 
\end{prop}

\begin{proof}
According to a theorem of Suslin (see \cite[Theorem 4.7.1]{MR3972198}), the field $K^{(\infty)}$ has cohomological dimension at most $2$ if and only if for all finite extensions $K^{(\infty)}\subset L^{(\infty)}$ and for all central simple algebras $A$ over $L^{(\infty)}$, the reduced norm $\text{Nrd}:A^*\rightarrow (L^{(\infty)})^*$ is surjective. In particular, proving that $\text{H}^1(L^{(\infty)}, SL_1(A))\simeq \text{H}^1(K^{(\infty)}, R_{L^{(\infty)}/K^{(\infty)}}(SL_1(A)))$ is trivial will establish our claim because then we will have the following exact sequence:
\begin{equation*}
        \begin{tikzpicture}[baseline= (a).base]
        \node[scale=1] (a) at (0,0){
        \begin{tikzcd}
       1 \arrow[r]& SL_1(A)\arrow[r] & A^*\arrow[r, "\text{Nrd}"] & (L^{(\infty)})^*\arrow[r] & 1. 
       \end{tikzcd}};
        \end{tikzpicture}
\end{equation*}
Let $Y\in \text{H}^1(L^{(\infty)}, SL_1(A))$. Select $\alpha \in L^{(\infty)}$ such that $L^{(\infty)}= K^{(\infty)}(\alpha)$. We can, without loss of generality, replace $K$ with $K^{(n)}$ for a sufficiently large $n$ to assume that $A$ and $Y$ are defined over $L\coloneqq K(\alpha)$ (i.e. $A=\Tilde{A}\otimes_L L^{(\infty)}$ and $Y=\Tilde{Y}\times_L L^{(\infty)}$) and that $[L^{(\infty)}:K^{(\infty)}]=[L:K]$. Suppose, by contraction, that $Y$ has no $L^{(\infty)}$-points. The Weil restriction $R_{L/K}\Tilde{Y}$ is a $R_{L/K}(SL(\Tilde{A}))$-torsor over $K$. Consequently, by construction, the quotient field $\Tilde{K}\coloneqq k(R_{L/K}\Tilde{Y})$ is contained in $K^{(1)}\subset K^{(\infty)}$. The isomorphism between $L\otimes_K \Tilde{K}$ and $\Tilde{L}\coloneqq \Tilde{K}(\alpha)$ induces an isomorphism between $(R_{L/K}\Tilde{Y})_{\Tilde{K}}$ and $R_{\Tilde{L}/\Tilde{K}}(\Tilde{Y}_{\Tilde{L}})$. By construction, $R_{L/K}\Tilde{Y}$ has a $\Tilde{K}$-point. Consequently, $R_{\Tilde{L}/\Tilde{K}}(\Tilde{Y}_{\Tilde{L}})$ also has a $\Tilde{K}$-point. This implies that $\Tilde{Y}$ has a $\Tilde{L}$-point, leading to a contradiction since $\Tilde{L}\subset L^{(\infty)}$.
\end{proof}

\begin{oss}
The cohomological dimension of $K^{(\infty)}$ could be strictly smaller than $2$. For example, if $K$ is algebraically closed, then $K=K^{(1)}=K^{(\infty)}$, so the cohomological dimension is $0$. We note that $K=K^{(1)}=K^{(\infty)}$ holds for all fields of cohomological dimension $\leq 2$, by Suslin's Theorem (see \cite[Theorem 4.7.1]{MR3972198}). 

However, if we assume that $\text{Br}(K)$ is not trivial, then Lemma \ref{lemmacohomo2} implies that $\text{Br}(K^{(\infty)})$ is not trivial. According to \cite[Theorem 6.1.8]{gille_szamuely_2006}, this implies that the cohomological dimension of $K^{(\infty)}$ is at least $2$ and, therefore, it is precisely $2$.
\end{oss}

\section{Construction of the homogeneous space when $K=k((t))$ for some local field $k$}
\label{localt}

In this section we prove Theorem \ref{thm2}. We are going to construct a homogeneous space over $K=k((t))$, for some local field $k$. Thanks to the observation made at the end of Section \ref{sezionefla}, it is sufficient to find a suitable connected, reductive, linear $K$-algebraic group $H$ as in Proposition \ref{commuta}, that is:
\begin{itemize}
    \item the maximal toric quotient $H^{\text{tor}}$ is flasque;
    \item the map on cohomology $\Hh^1(K, H^{\text{tor}})\rightarrow \Hh^2(K,Z^{\text{(ss)}})\rightarrow \Hh^3(K,\mu)$ is non-trivial, where the first map is the connecting morphism induced by:
    \begin{equation*}
        \begin{tikzpicture}[baseline= (a).base]
        \node[scale=1] (a) at (0,0){
        \begin{tikzcd}
      1 \arrow[r] & Z^{\text{(ss)}} \arrow[r] & Z \arrow[r] & H^{\text{tor}} \arrow[r] & 1.
       \end{tikzcd}};
        \end{tikzpicture}
    \end{equation*}
    and the second one is the connecting map induced by:
    \begin{equation*}
        \begin{tikzpicture}[baseline= (a).base]
        \node[scale=1] (a) at (0,0){
        \begin{tikzcd}
      1 \arrow[r] & \mu \arrow[r] & Z^{\text{(sc)}} \arrow[r] & Z^{\text{(ss)}} \arrow[r] & 1.
       \end{tikzcd}};
        \end{tikzpicture}
\end{equation*}
\end{itemize}

The section is organised as follows. First of all, we reduce the construction of $H$ to the existence of an extension with `nice' properties in a similar way of Section \ref{dim2}, without assuming any condition on the field $K$. In the following subsections, we are going to prove the existence of such extension.

\begin{prop}
\label{extensione}
Consider an extension of groups of multiplicative type over $K$:
\begin{equation*}
        \begin{tikzpicture}[baseline= (a).base]
        \node[scale=1] (a) at (0,0){
        \begin{tikzcd}
       1 \arrow[r] & \mu_p\arrow[r] & R \arrow[r]& F\arrow[r] & 1
       \end{tikzcd}};
        \end{tikzpicture}
\end{equation*}
that satisfies the following properties: 
\begin{itemize}
    \item $F$ is an algebraic torus,
    \item the composition of connecting morphisms $\Hh^1(K,F)\rightarrow \Hh^2(K,\mu_p)\rightarrow \Hh^3(K,\mu_p)$ is non-trivial, where the second map is induced in cohomology by the Bockstein exact sequence:
    \begin{equation*}
        \begin{tikzpicture}[baseline= (a).base]
        \node[scale=1] (a) at (0,0){
        \begin{tikzcd}
      1 \arrow[r] & \mu_{p} \arrow[r] & \mu_{p^2} \arrow[r, "(-)^{p}"] & \mu_{p} \arrow[r] & 1.
       \end{tikzcd}};
        \end{tikzpicture}
\end{equation*}
\end{itemize}
Then, there exists a connected, reductive, linear $K$-algebraic group $H$ such that:
\begin{itemize}
    \item the maximal toric quotient of $H$ is $F$,
    \item the map on cohomology $\Hh^1(K,F){\longrightarrow} \Hh^2(K,Z^{(\emph{ss})}){\longrightarrow} \Hh^3(K,\mu)$
is non-trivial.
\end{itemize}
\end{prop}

\begin{proof}
We consider the following central exact sequence:
\begin{equation*}
        \begin{tikzpicture}[baseline= (a).base]
        \node[scale=1] (a) at (0,0){
        \begin{tikzcd}
       1 \arrow[r] & \mu_p\arrow[r] & SL_{p^2}/\mu_p \arrow[r]& PSL_{p^2}\arrow[r] & 1.
       \end{tikzcd}};
        \end{tikzpicture}
\end{equation*}
We define the linear $K$-algebraic group $H$ to fit into the following exact sequence:
\begin{equation*}
        \begin{tikzpicture}[baseline= (a).base]
        \node[scale=1] (a) at (0,0){
        \begin{tikzcd}
       1 \arrow[r] & \mu_p\arrow[r] & SL_{p^2}/\mu_p\times_K R \arrow[r]& H\arrow[r] & 1,
       \end{tikzcd}};
        \end{tikzpicture}
\end{equation*}
as in the proof of Proposition \ref{main1}. The group $H$ is connected, reductive and its maximal toric quotient is $F$. By construction $H^{(\text{ss})}$ is $SL_{p^2}/\mu_p$, $H^{(\text{sc})}$ is $SL_{p^2}$ and $\mu$ is $\mu_p$. The centers $Z^{(\text{sc})}$ and $Z^{(\text{ss})}$ are $\mu_{p^2}$ and $\mu_p$, respectively. The extension given by 
\begin{equation*}
        \begin{tikzpicture}[baseline= (a).base]
        \node[scale=1] (a) at (0,0){
        \begin{tikzcd}
       1 \arrow[r] & \mu\arrow[r] & Z^{(\text{sc})} \arrow[r]&Z^{(\text{ss})} \arrow[r] & 1
       \end{tikzcd}};
        \end{tikzpicture}
\end{equation*}
coincides with the Bockstein exact sequence. This leads to the conclusion.
\end{proof}

Furthermore, for the next example we would like that the algebraic torus $F$ satisfy the following properties:
\begin{itemize}
    \item $F$ is a flasque torus,
    \item the group $\Hh^1(K,F)$ is isomorphic to $\mathbb Z/ p \mathbb Z$.
\end{itemize}
This is sufficient to obtain a homogeneous space whose evaluation map is zero by Proposition \ref{commuta}.

\subsection{Duality theorems on higher local fields}

A prerequisite for constructing an extension, as described in Proposition \ref{extensione}, is the requirement that the Bockstein connecting map from $\Hh^2(K,\mu_p)$ to $\Hh^3(K,\mu_p)$ is non-trivial. Therefore, we will select a suitable $K$ in order to ensure that this condition is satisfied. Furthermore, the connecting map will be surjective and $\Hh^3(K,\mu_p)$ will be non-zero.

In this subsection, we assume that $K$ is a $2$-local field, that means there is a chain of fields $K_2=K$, $K_1$ and $K_0$ such that $K_{i+1}$ is a complete discrete valuation field with residue field $K_i$ and $K_0$ is finite, and the characteristic of $K_1$ is not $p$. First of all we recall some notations and facts about duality, for a more detailed discussion see \cite[p. 37-39]{milne2006}.

Firstly, we recall some notations. The $\Gamma_K$-module $\mu_{p^{\infty}}$ is the colimit $\varinjlim\mu_{p^n}$ of the $p^n$-th roots of unity (i.e. the union for of all $n$ of the $p^n$-th roots of unity). For a finite and $p$-primary $\Gamma_K$-module $\mathcal M$, we denote by $\mathcal M^{\otimes 2}$ the tensor product ${\mathcal M}\otimes_{\mathbb Z} {\mathcal M}$ with the diagonal action of $\Gamma_K$, and ${\mathcal M}^*$ denotes the group $\Hom_{\mathbb Z}({\mathcal M}, \mu_{p^{\infty}}^{\otimes 2})$ with the following $\Gamma_K$ action: $\sigma \cdot f\coloneqq \sigma \circ f \circ \sigma^{-1}$. By combining well-known results, we obtain the following theorem:

\begin{thm}
\label{kato}
We have the following properties:
\begin{itemize}
    \item There exists an isomorphism $\Hh^3(K, \mu_{p^n}^{\otimes 2})\simeq \mathbb Z/p^n \mathbb Z$.
    \item The cup-product pairing
    \begin{equation*}
        \begin{tikzpicture}[baseline= (a).base]
        \node[scale=1] (a) at (0,0){
        \begin{tikzcd}
       \Hh^r(K, {\mathcal M}) \times \Hh^{3-r}(K, {\mathcal M}^*) \arrow[r] & \Hh^3(K, \mu_{p^{\infty}}^{\otimes 2})\simeq \mathbb Q_p/\mathbb Z_p
       \end{tikzcd}};
        \end{tikzpicture}
\end{equation*}
    is a non-degenerate pairing of finite groups for all $r=0,1,2,3$.
\end{itemize}
\end{thm}

\begin{proof}
The first item is \cite[Theorem III]{10.3792/pjaa.54.250}, and the second one is \cite[Theorem 2.17]{milne2006}.
\end{proof}

As a direct consequence, we can now proceed to prove the following lemma:

\begin{lemma}
\label{pulizia}
Let $K$ be a $2$-local field. Assume that $K$ contains a primitive $p$-th root of unity and it does not contain a  primitive $p^2$-th root of unity. Then, the connecting map $\delta: \Hh^2(K, \mu_p)\rightarrow \Hh^3(K,\mu_p)$, induced by the Bockstein extension
\begin{equation*}
        \begin{tikzpicture}[baseline= (a).base]
        \node[scale=1] (a) at (0,0){
        \begin{tikzcd}
      1 \arrow[r] & \mu_{p} \arrow[r, "i"] & \mu_{p^2} \arrow[r, "(-)^{p}"] & \mu_{p} \arrow[r] & 1,
       \end{tikzcd}};
        \end{tikzpicture}
\end{equation*}
is non-zero.
\end{lemma}

\begin{proof}
The functor $\text{Hom}_{\mathbb Z}(-, \mu_{p^{\infty}}^{\otimes 2})$ defines an anti-equivalence of category from the category of finite $p$-primary $\Gamma_K$-modules to itself, that preserves exact sequences. In fact, this functor operates as $\text{Hom}_{\mathbb Z}(-, \mathbb Q_p/\mathbb Z_p)$ over the group structure. The dual sequence of the extension in the statement is itself. Consequently, we have the following commutative diagram:
\begin{equation*}
        \begin{tikzpicture}[baseline= (a).base]
        \node[scale=1] (a) at (0,0){
        \begin{tikzcd}
       0\ \ \ \ \ \ \ 0 \arrow[d,shift right=3.3ex] \\ 
      \mu_p \times \mu_p \arrow[r]\arrow[d,shift right=3.3ex, "i"]\arrow[u, shift left=-3.3ex] & \mu_{p^{\infty}}^{\otimes 2}\arrow[d, "\id", "\wr"']\\
      \mu_{p^2}\times \mu_{p^2} \arrow[r]\arrow[d,shift right=3.3ex, "(-)^p"]\arrow[u, shift left=-3.3ex, "(-)^p"'] & \mu_{p^{\infty}}^{\otimes 2}\arrow[d, "\id", "\wr"']\\
       \mu_p \times  \mu_p \arrow[r]\arrow[d,shift right=3.3ex]\arrow[u, shift left=-3.3ex, "i"'] & \mu_{p^{\infty}}^{\otimes 2}\\
      0\ \ \ \ \ \ \ 0. \arrow[u, shift left=-3.3ex]
       \end{tikzcd}};
        \end{tikzpicture}
\end{equation*}
There is the following commutative diagram, up to a sign:
\begin{equation*}
        \begin{tikzpicture}[baseline= (a).base]
        \node[scale=1] (a) at (0,0){
        \begin{tikzcd}
       \Hh^2(K, \mu_p) \times \Hh^1(K,  \mu_p) \arrow[d,shift right=6.50ex, "\delta"]\arrow[r] & \Hh^3(K, \mu_{p^{\infty}}^{\otimes 2})\simeq \mathbb Q_p / \mathbb Z_p\arrow[d, "\id", "\wr"']\\ 
      \Hh^3(K,  \mu_p) \times \Hh^0(K,  \mu_p) \arrow[r]\arrow[d,shift right=6.50ex, "i_*"]\arrow[u, shift left=-7.75ex] & \Hh^3(K, \mu_{p^{\infty}}^{\otimes 2})\simeq \mathbb Q_p / \mathbb Z_p\arrow[d, "\id", "\wr"']\\
      \Hh^3(K, \mu_{p^2}) \times \Hh^0(K, \mu_{p^2}) \arrow[r]\arrow[u, shift left=-7.75ex, "(-)^p_*"'] & \Hh^3(K, \mu_{p^{\infty}}^{\otimes 2})\simeq \mathbb Q_p / \mathbb Z_p,
       \end{tikzcd}};
        \end{tikzpicture}
\end{equation*}
where the top square commutes because of \cite[Proposition 3.4.9]{gille_szamuely_2006}, while the second one commutes because of the functoriality of the construction of the cup-product. By Theorem \ref{kato}, the cup-product pairing is non-degenerate of finite $p$-groups. Furthermore, the pairing is perfect in view of the isomorphism $\Hh^3(K,\mu_{p^{\infty}}^{\otimes 2})\simeq \mathbb Q_p / \mathbb Z_p$. Since $K$ does not contain a primitive $p^2$-th root of unity the map $(-)^p:\mu_{p^2}(K)\rightarrow \mu_p(K)$ is zero. Since the pairing is perfect, the map $i_*:\Hh^3(K,\mu_p)\rightarrow \Hh^3(K,\mu_{p^2})$ is zero as well, and therefore, the map $\delta$ is surjective. Since $\mu_p\subset K$ there is a chain of isomorphism $\Hh^3(K,\mu_p)=\Hh^3(K,\mu_p^{\otimes2})\simeq \mathbb Z/ p \mathbb Z$, so we can conclude that $\delta$ is non-zero.
\end{proof}

\subsection{Construction of the flasque torus}

In this subsection we construct the flasque torus $F$ and establish that, for certain specific fields, the firt cohomology group is isomorphic to $\mathbb Z/p \mathbb Z$. We begin with a lemma.

\begin{lemma}
\label{esatta}
Let $\mathcal F$ be the family of maximal proper subgroups of a non-cyclic abelian group $\Gamma$. Let $Q$ be a coflasque $\Gamma$-module, and let $\psi:\mathbb Z \rightarrow Q$ be a morphism of $\Gamma$-modules. Consider the following exact sequence of $\Gamma$-modules:
\begin{equation*}
        \begin{tikzpicture}[baseline= (a).base]
        \node[scale=1] (a) at (0,0){
        \begin{tikzcd}
      0 \arrow[r]& \mathbb Z \arrow[r, "(\Delta{,}  \psi )"] \arrow[r]& \bigoplus_{\Lambda \in \mathcal F} \mathbb Z[\Gamma/\Lambda] \oplus Q \arrow[r] & N \arrow[r] & 0
       \end{tikzcd}};
        \end{tikzpicture}
\end{equation*}
where $\Delta$ is the diagonal embedding. Then $N$ is coflasque. 
\end{lemma}

\begin{proof}
First of all, we note that $N$ is torsion-free since the sequence splits as abelian groups. For all $\Tilde{\Gamma} < \Gamma$, there is the following long exact sequence:
\begin{equation*}
        \begin{tikzpicture}[baseline= (a).base]
        \node[scale=1] (a) at (0,0){
        \begin{tikzcd}[column sep=large]
     \Hh^1(\Tilde{\Gamma},\bigoplus_{\Lambda \in \mathcal F} \mathbb Z[\Gamma/\Lambda] \oplus Q ) \arrow[r] & \Hh^1(\Tilde{\Gamma}, N) \arrow[r]& \Hh^2(\Tilde{\Gamma}, \mathbb Z) \arrow[r,  "(\Delta_*{,} \psi_*)"]& \bigoplus_{\Lambda \in \mathcal F} \Hh^2(\Tilde{\Gamma}, \mathbb Z[\Gamma/\Lambda] \oplus Q).
       \end{tikzcd}};
        \end{tikzpicture}
\end{equation*}
The $\Tilde{\Gamma}$-module $\bigoplus_{\Lambda \in \mathcal F} \mathbb Z[\Gamma/\Lambda] \oplus Q$ is coflasque, thus, to prove the claim, it is sufficient to demonstrate that the map
\begin{equation*}
        \begin{tikzpicture}[baseline= (a).base]
        \node[scale=1] (a) at (0,0){
        \begin{tikzcd}[column sep=large]
     \Hh^2(\Tilde{\Gamma}, \mathbb Z) \arrow[r,  "\Delta_*"]& \bigoplus_{\Lambda \in \mathcal F} \Hh^2(\Tilde{\Gamma}, \mathbb Z[\Gamma/\Lambda]) 
       \end{tikzcd}};
        \end{tikzpicture}
\end{equation*}
    is injective. 
    \begin{itemize}
        \item If $\Tilde{\Gamma}$ is a proper subgroup of $\Gamma$, there exists a maximal proper subgroup $\Lambda$ of $\Gamma$ that contains $\Tilde{\Gamma}$. Since the subgroup $\Tilde{\Gamma}$ acts trivially on $\mathbb Z[\Gamma/\Lambda]$, the $\Gamma$-module $\mathbb Z[\Gamma/\Lambda]$ is isomorphic to $\mathbb Z^{[\Gamma:\Lambda]}$ as a $\Tilde{\Gamma}$-module, and the induced map on cohomology
        $$\Hh^2(\Tilde{\Gamma}, \mathbb Z)\overset{\Delta_*}{\longrightarrow} \Hh^2(\Tilde{\Gamma}, \mathbb Z)^{[\Gamma:\Lambda]}$$
        is injective.
        \item In the case of $\Gamma$ itself, we can rewrite 
        $$\Hh^2(\Gamma, \mathbb Z) \overset{\Delta_*}{\longrightarrow} \bigoplus_{\Lambda \in \mathcal F} \Hh^2(\Gamma, \mathbb Z[\Gamma/\Lambda]) $$
        as
        $$\Hom(\Gamma, \mathbb Q / \mathbb Z)= \Hh^1(\Gamma, \mathbb Q / \mathbb Z) \overset{\text{Res}}{\longrightarrow} \bigoplus_{\Lambda \in \mathcal F} \Hh^1(\Lambda, \mathbb Q /\mathbb Z) = \bigoplus_{\Lambda \in \mathcal F}\text{Hom}(\Lambda, \mathbb Q / \mathbb Z).$$
        An element $\chi$ belongs to the kernel of this map if and only if $\chi_{|\Lambda}$ is trivial for all $\Lambda\in \mathcal F$. Since $\Gamma$ is not cyclic, the union of subgroups in $\mathcal F$ covers the entire group $\Gamma$. In particular, $\text{Res}$ must be injective.
    \end{itemize}
\end{proof}

\begin{oss}
Previous lemma holds even in the case where $\Gamma$ is non-cyclic nilpotent, since all the elements of $\mathcal F$ are normal subgroups (see \cite[Proposition 5.2.4]{robinson1996course}). 
\end{oss}

Our main application of previous lemma is the following proposition:

\begin{prop}
\label{piatto}
Suppose that $\Gamma=(\mathbb Z/ p \mathbb Z)^2$. Let $K\subset K'$ be a Galois extension with group $\Gamma$. Let $\{K_i\}_{i=1}^{p+1}$ be the family of degree $p$ subextensions of $K\subset K'$. Consider the following exact sequence:
\begin{equation*}
        \begin{tikzpicture}[baseline= (a).base]
        \node[scale=1] (a) at (0,0){
        \begin{tikzcd}
      1 \arrow[r] & \mu_p \arrow[r] & \prod_{i=1}^{p+1} R_{K_i/K}(\mathbb G_{m,K_i}) \arrow[r] & F \arrow[r] & 1
       \end{tikzcd}};
        \end{tikzpicture}
\end{equation*}
where the first map is the diagonal embedding. Then $F$ is a flasque torus.
\end{prop}

\begin{proof}
The character module of $R_{K_i/K}(\mathbb G_{m,K_i})$ is $\mathbb Z[\Gamma_K/\Gamma_{K_i}]\simeq \mathbb Z[\Gamma/\Lambda_i]$, where $\Lambda_i$ is the subgroup of $\Gamma$ that fixes $K_i$. The group $F$ is the quotient of a $K$-algebraic torus, so it is a $K$-algebraic torus as well. There is a dual exact sequence of $\Gamma_K$-modules (see \cite[Theorem 7.3]{Affinegroup}):
 \begin{equation*}
        \begin{tikzpicture}[baseline= (a).base]
        \node[scale=1] (a) at (0,0){
        \begin{tikzcd}
      0 \arrow[r] & \widehat F \arrow[r] &\bigoplus_{i=1}^{p+1} \mathbb Z[\Gamma_K/\Gamma_{K_i}] \arrow[r, "\epsilon"]& \mathbb Z/p\mathbb Z \arrow[r] & 0,
       \end{tikzcd}};
        \end{tikzpicture}
\end{equation*}
where $\epsilon(\gamma \Gamma_{K_i})=1$ for all $\gamma\in \Gamma_K$ and $i=1,\cdots, p+1$. The action of $\Gamma_K$ on the elements of the exact sequence factors trough the action of $\Gamma$, so we can consider it as a sequence of $\Gamma$-modules. The following diagram is cartesian 
\begin{equation*}
        \begin{tikzpicture}[baseline= (a).base]
        \node[scale=1] (a) at (0,0){
        \begin{tikzcd}
      \bigoplus_{i=1}^{p+1} \mathbb Z[\Gamma/\Lambda_i] \arrow[r, "\epsilon"]& \mathbb Z/p\mathbb Z \\
      \bigoplus_{i=1}^{p+1} \mathbb Z[\Gamma/\Lambda_i] \oplus \mathbb Z \arrow[r, "(\epsilon{,} p\cdot)"]\arrow[u, "\pi_1"]& \mathbb Z  \arrow[u, "\pi_2"]
       \end{tikzcd}};
        \end{tikzpicture}
\end{equation*}
where $p\cdot$ is the multiplication by $p$. Thus, we have the following exact sequence of $\Gamma$-modules:
\begin{equation*}
        \begin{tikzpicture}[baseline= (a).base]
        \node[scale=1] (a) at (0,0){
        \begin{tikzcd}
      0 \arrow[r]& \widehat F \arrow[r]& \bigoplus_{i=1}^{p+1} \mathbb Z[\Gamma/\Lambda_i] \oplus \mathbb Z \arrow[r, "(\epsilon{,} p\cdot)"]& \mathbb Z \arrow[r]& 0.
       \end{tikzcd}};
        \end{tikzpicture}
\end{equation*}
We can take the dual sequence of $\Gamma$-lattices:
\begin{equation*}
        \begin{tikzpicture}[baseline= (a).base]
        \node[scale=1] (a) at (0,0){
        \begin{tikzcd}
      0 \arrow[r]&  \mathbb Z\arrow[r, "(\Delta{,} p\cdot)"]& \bigoplus_{i=1}^{p+1} \mathbb Z[\Gamma/\Lambda_i] \oplus \mathbb Z \arrow[r]& \widehat F^0 \arrow[r]& 0.
       \end{tikzcd}};
        \end{tikzpicture}
\end{equation*}
By Lemma \ref{esatta}, the $\Gamma$-module $\widehat F^0$ is coflasque and thus $F$ is a flasque torus.
\end{proof}

For certain specific fields $K$ we can completely determine $\Hh^1(K,F)$. Let us assume that $K=k((t))$ for some field $k$ of characteristic $0$. We begin with a lemma that summarises well known results on central simple algebras.

\begin{lemma}
\label{accorpamento}
With the above notations, the following properties hold: 
\begin{itemize}
    \item There is a split exact sequence
    \begin{equation*}
        \begin{tikzpicture}[baseline= (a).base]
        \node[scale=1] (a) at (0,0){
        \begin{tikzcd}
      0 \arrow[r]&  \emph{Br}(k)\arrow[r]& \emph{Br}(K) \arrow[r]& \Hom_{\emph{cont}}(\Gamma_{k}, \mathbb Q/\mathbb Z) \arrow[r]\arrow[l, bend left=30, start anchor={[xshift=-6.5ex]}]& 0,
       \end{tikzcd}};
        \end{tikzpicture}
    \end{equation*}
    where the first map is the scalar extension, and the section associates the cyclic algebra $(\chi,t)$ to a continuous character $\chi:\Gamma_{k}\rightarrow \mathbb Q /\mathbb Z$.
    \item The period of the cyclic algebra $(\chi,t)$ coincides with the cardinality of the (finite) image of $\chi$.
\end{itemize}
\end{lemma}

\begin{proof}
The first item is the combination of \cite[Corollary 6.3.5]{gille_szamuely_2006} and \cite[Remark 6.3.6]{gille_szamuely_2006}. We denote by $n$, the cardinality of the image of $\chi$. Let $k\subset k'$ be a cyclic finite Galois extension of degree $n$ that induces an isomorphism $\text{Gal}(k'/k)\rightarrow \mathbb Z/ n \mathbb Z$. We can determine the powers of $t$ that are contained within $N_{k'((t))/k}(k'((t))^*)$; they are precisely $t^{ni}$ for $i \in \mathbb Z$. To see this, consider an element $r=at^i\sum_{j=0}^{\infty} a_jt^j$, where $a_j \in k'$ and $a_0=1$, in $k'((t))$. Then,
$$N(r)=N(a)t^{ni}\prod_{\sigma \in \text{Gal}(k'/k)}\left(\sum_{j=0}^{\infty} \sigma(a_j)t^j\right)=N(a)t^{ni}(1+\dots).$$
By \cite[Corollary 4.7.5]{gille_szamuely_2006}, we can conclude that the period of $(\chi, t)$ is $n$. 
\end{proof}

We assume that $k$ contains a primitive $p$-th root of unity. Let $k\subset k'$ be a Galois extension of degree $p$ and let $\chi: \text{Gal}(k'/k)\rightarrow \mathbb Z/ p \mathbb Z$ be an isomorphism.

\begin{lemma}
\label{lemmasplit}
Using the above notations, consider a central simple algebra $A$ over $K$. If $A$ is split by $K(t^{1/p})$ and by $k'((t))$, then $A$ is Brauer equivalent to a power of the cyclic algebra $(\chi,t)$.  
\end{lemma}

\begin{proof}
By Lemma \ref{accorpamento}, we can deduce that there exists a central simple algebra 
$C$ defined over $k$ and a continuous character $\psi: \Gamma_{k}\rightarrow \mathbb Q / \mathbb Z$ such that $A$ is Brauer equivalent to the tensor product $(C\otimes_{k} K) \otimes_{K} (\psi, t) $. Furthermore, the condition for $A$ to be split over a field extension is equivalent to both $C\otimes_{k}K $ and $(\psi,t)$ being split over the same field extension. Since $A$ is split over a degree $p$ extension, its order must divide $p$, and the same holds for $(\psi,t)$. By Lemma \ref{accorpamento}, the extension map $\text{Br}(k)\rightarrow \text{Br}(K(t^{1/p}))$ is an embedding. In particular $C\otimes_{k}K(t^{1/p})$ is split if and only if $C$ is split. It remains to demonstrate that $(\psi,t)$ is a power of $(\chi,t)$. In other words, we need to show that the field extension associated with $\psi$ is contained within the field extension associated with $\chi$. Suppose that this is not true, then by Lemma \ref{accorpamento} the field extension associated to $\psi$ has degree $p$, and the composition of the two field extensions has degree $p^2$. The algebra $(\psi,t)\otimes_{K}k'((t))$ is isomorphic to $(\psi', t)$, where $\psi':\Gamma_{k'}\rightarrow \mathbb Q/ \mathbb Z$ is a non-trivial character. However, Lemma \ref{accorpamento} implies that the cyclic algebra $(\psi',t)$ has order $p$. Consequently, it cannot be split, contradicting the initial hypothesis of the statement.
\end{proof}

Let $K'$ be the field $k'((t^{1/p}))$. The field extension $K\subset K'$ is Galois of group $(\mathbb Z/p \mathbb Z)^2$. The next lemma constructs `nice' splitting fields of the chosen cyclic algebra $(\chi,t)$.

\begin{lemma}
\label{annullamento}
Assume that $p\neq 2$. Then, the $K$-cyclic algebra $(\chi, t)$ splits over all degree $p$ subextensions of $K\subset k'((t^{1/p}))$.
\end{lemma}

\begin{proof}
The extension $K\subset k'((t^{1/p}))$ is Galois with group $(\mathbb Z/ p\mathbb Z)^2$. By Kummer's Theorem (see \cite[Theorem 6.2]{lang02}), there exists $a \in k'$ such that $a^p\in k$ and $k'=k(a)$. By Kummer theory, all the intermediate subextensions of $K\subset k'((t^{1/p}))=K(a, t^{1/p})$ are of the form $K_{[i,j]}\coloneqq K(a^i t^{j/p})$ for $[i,j]\in \mathbb P^1(\mathbb Z/p\mathbb Z)$. Fix a primitive $p$-th root of unity $\omega$. By \cite[Corollary 2.5.5]{gille_szamuely_2006}, there is an isomorphism of $K$-algebras $(a^p,t)_{\omega}\simeq (\chi,t)$. By \cite[Proposition 4.7.1]{gille_szamuely_2006}, there is a functorial morphism (it is an isomorphism by Merkurjev-Suslin Theorem, but it not necessary for this proof):
$$K^M_2(K)/p{\longrightarrow}  \Hh^2(K, \mu_p^{\otimes 2})=\Hh^2(K, \mu_p)$$
that sends the symbol $\{a^p,t\}$ to the class of cyclic algebra $(a^p, t)_{\omega}$. We prove the statement by working with symbols. There are the following equalities over $K$: 
$$\{a^p, a^{ip}t^{j}\}= \{a^p, a^{ip}\}+ \{a^p, t^j\}=i\{a^p, a^{p}\}+ j\{a^p, t\} = i\{a^p, -1\}+ j\{a^p, t\}=$$
$$=i\{a^p, (-1)^p\}+ j\{a^p, t\}=ip\{a^p, -1\}+ j\{a^p, t\}=j\{a^p, t\},$$
If $j$ is coprime to $p$, then there exist $m$ and $n$ such that $nj+mp=1$, and so:
$$\{a^p, t\}= nj\{a^p, t\}+mp\{a^p,t\}=n\{a^p, a^{ip}t^{j}\}.$$
If we extend the scalar to $K_{[i,j]}$, then we get:
$$\{a^p, t\}=n\{a^p, a^{ip}t^{j}\}=np\{a^p, a^{i}t^{j/p}\}=0.$$
If $[i,j]= [i,0]$, then $K_{[i,0]}= k'((t))$ and thus $\{a^p, t\}= p\{a,t\}=0$.
\end{proof}

\begin{oss}
\label{p2}
A couple of comments about the case $p=2$:
\begin{itemize}
    \item If we assume that $k$ contains a primitive $4$-th root of unity $\zeta$, then Lemma \ref{annullamento} holds also for $p=2$. The proof is the same, except for the equality $\{a^p,-1\}=\{a^p,(-1)^p \}$, which needs to be replaced with $\{a^p,-1 \}=\{a^p, {\zeta}^p\}$. 
    \item On the contrary, if $-1$ is not a square in $k$, then the lemma is no longer true. For example we can take $k=\mathbb R$ and $k'=\mathbb C$. If we choose $(-1,t)$ as cyclic algebra over $ K$, then $K((-t)^{1/2})$ does not split $(-1,t)$. Indeed over $K((-t)^{1/2})$ the algebra $(-1,t)$ is isomorphic to $(-1,-1)$. Now, it is not difficult to show that $(-1,-1)$ is a division algebra over $K((-t)^{1/2})$, for example (see \cite[Proposition 1.1.7]{gille_szamuely_2006}) because the equation $x^2+y^2+z^2+w^2=0$ has only the trivial solution over $K((-t)^{1/2})=\mathbb R(((-t)^{1/2}))$.
\end{itemize}
\end{oss}

We can consider a flasque torus $F$ as in Proposition \ref{piatto}. In the following proposition, we will determinate the cohomology group $\Hh^1(K,F)$.

\begin{prop}
\label{pz}
Given the fields as described above, the group $\Hh^1(K, F)$ is either trivial or isomorphic to $\mathbb Z/p\mathbb Z$. In particular, if $p\neq 2$, then $\Hh^1(K,F)$ is isomorphic to $\mathbb Z/p\mathbb Z$.
\end{prop}

\begin{proof}
We have the following exact sequence of cohomology groups:
\begin{equation*}
        \begin{tikzpicture}[baseline= (a).base]
        \node[scale=0.99] (a) at (0,0){
        \begin{tikzcd}
        \Hh^1\left(K, \prod_{i=1}^{p+1} R_{K_i/K}(\mathbb G_{m,K_i})\right) \arrow[r] & \Hh^1(K,F) \arrow[r] & \Hh^2(K,\mu_p)\arrow[r] & \Hh^2\left(K, \prod_{i=1}^{p+1} R_{K_i/K}(\mathbb G_{m,K_i})\right). 
       \end{tikzcd}};
        \end{tikzpicture}
\end{equation*}
The first group is trivial by Shapiro's Lemma and Hilbert `90. The group $\Hh^2(K,\mu_p)$ is the $p$-torsion of the Brauer group, and the group $\Hh^2\left(K, \prod_{i=1}^{p+1} R_{K_i/K}(\mathbb G_{m,K_i})\right)$ is isomorphic to $ \prod_{i=1}^{p+1}\text{Br}(K_i)$. Then, $\Hh^1(K,F)$ is the group of division algebras of order dividing $p$ that are split over all the $K_i$. Let $A$ be a central simple algebra of order dividing $p$ over $K$, which also splits over all the field extensions $K_i$. By Lemma \ref{lemmasplit}, such an algebra is Brauer equivalent to a power of the cyclic algebra $(\chi,t)$. This implies that $\Hh^1(K,F)$ is a subgroup of the group $\guilsinglleft(\chi,t)\guilsinglright\simeq \mathbb Z/p\mathbb Z$. Therefore, it can be either trivial or isomorphic to $\mathbb Z/p\mathbb Z$. If $p\neq 2$, then $(\chi,t)$ splits over all the extensions $K_i$ by Lemma \ref{annullamento}, so $\Hh^1(K,F)$ is not trivial.
\end{proof}

\begin{oss}
If we assume that $k$ contains a primitive $4$-th root of unity, then in Proposition \ref{pz}, even for $p=2$, we get that $\Hh^1(K,F)$ is isomorphic to $\mathbb Z/2\mathbb Z$, in view of Remark \ref{p2}. 
\end{oss}

\subsection{Conclusion}

In this subsection, we prove the exact sequence of Proposition \ref{piatto} verifies the conditions of the beginning of Section \ref{localt}. This will finish the construction of the homogeneous space. In the first part, we only assume that $K$ is the field $k((t))$, for some characteristic $0$ field $k$. We assume that $k$ contains a primitive $p$-th root of unity, for some prime $p\neq 2$.

\begin{prop}
\label{quasifine}
Assume that $k$ has $p$-cohomological dimension at most $2$. Let $A$ be a central simple algebra of period $p$ over $K$, such that the Brauer class of $A$ does not belong to the image of $\emph{Br}(k)\rightarrow \emph{Br}(K)$. Then, there exists a cyclic algebra $(\chi,t)$ of period $p$ over $K$ satisfying the following:
\begin{itemize}
    \item Given the connecting map $\delta:\Hh^2(K,\mu_p)\rightarrow \Hh^3(K,\mu_p) $, induced by the Bockstein exact sequence, there is an identity $\delta([A])=\delta ([\chi,t])$.
    \item The Brauer class of $(\chi,t)$ belongs to the image of the connecting map $\Hh^1(K,F)\rightarrow \Hh^2(K,\mu_p)$ induced by an extension like the one in Proposition \ref{piatto}.    
\end{itemize}
\end{prop}

\begin{proof}
According to Lemma \ref{accorpamento}, there exists a central simple algebra $C$ defined over $k$ and a continuous character $\chi: \Gamma_{k}\rightarrow \mathbb Q / \mathbb Z$ such that $A=[(C\otimes_{k} K) \otimes_{K} (\chi, t) ]$. Furthermore, since $A$ has period $p$ and its Brauer class is not in the image of $\text{Br}(k)\rightarrow \text{Br}(K)$, the algebra $(\chi,t)$ has period $p$ and $C\otimes_{k}K$ has period dividing $p$. The Bockstein exact sequence induces the following surjective morphism:
\begin{equation*}
        \begin{tikzpicture}[baseline= (a).base]
        \node[scale=1] (a) at (0,0){
        \begin{tikzcd}
     \Hh^2(k,\mu_{p^2}) \arrow[r, "(-)^p"] & \Hh^2(k,\mu_{p}) \arrow[r, "\delta"] & \Hh^3(k,\mu_{p})=0
       \end{tikzcd}};
        \end{tikzpicture}
\end{equation*}
because $k$ has $p$-cohomological dimension at most $2$. Consequently, there exists a central simple algebra $D$ of period $p^2$ over $k$, such that $[D]^{-p}=C$, hence 
$$\Bigl[A\otimes_K (D\otimes_{k}K)^{\otimes p}\Bigr]=\Bigl[(\chi,t)\Bigr]$$
and $\delta([\chi,t])=\delta([A])$. By Lemma \ref{accorpamento}, the map $\chi$ factors trough an isomorphism $\text{Gal}(k'/k)\rightarrow \mathbb Z / p \mathbb Z$ for some cyclic degree $p$ Galois extensions $k \subset k'$. By Lemma \ref{annullamento}, all degree $p$ field subextensions of $K\subset k'((t^{1/p}))$ splits $(\chi,t)$. Let $\{K_i\}_{i=1}^{p+1}$ denote the family of degree $p$ subextensions of $K\subset k'((t^{1/p}))$. We consider the following exact sequence
\begin{equation*}
        \begin{tikzpicture}[baseline= (a).base]
        \node[scale=1] (a) at (0,0){
        \begin{tikzcd}
      1 \arrow[r] & \mu_p \arrow[r] & \prod_{i=1}^{p+1} R_{K_i/K}(\mathbb G_{m,K_i}) \arrow[r] & F \arrow[r] & 1
       \end{tikzcd}};
        \end{tikzpicture}
\end{equation*}
as in Proposition \ref{piatto}. By Lemma \ref{annullamento}, the Brauer class of $[(\chi,t)]$ is in the image of the connecting map in view of the following exact sequence:
\begin{equation*}
        \begin{tikzpicture}[baseline= (a).base]
        \node[scale=1] (a) at (0,0){
        \begin{tikzcd}
      \Hh^1(K,F)\arrow[r] & \Hh^2(K,\mu_p)\arrow[r] & \prod_{i=1}^{p+1} \text{Br}(K_i).
       \end{tikzcd}};
        \end{tikzpicture}
\end{equation*}
This completes the proof.
\end{proof}

We can conclude the construction. Assume that $k$ is a local field and that $k$ does not contain a primitive $p^2$-th root of unity. By Lemma \ref{pulizia}, we can choose an element in $\Hh^2(K,\mu_p)$ that is not mapped to zero by the connecting map induced by the Bockstein sequence. By Proposition \ref{quasifine}, there is a cyclic algebra that is also not mapped to zero and it is in the image of the connecting map induced by an extension like the one in Proposition \ref{piatto}. By Proposition \ref{pz}, the group $\Hh^1(K,F)$ is isomorphic to $\mathbb Z/p\mathbb Z$. We have covered all the requested properties, thus we have just proved Theorem \ref{thm2}.

\printbibliography[
heading=bibintoc,
title={References}
]

\Addresses

\end{document}